\documentclass[twocolumn, final]{svjour3}
\usepackage{graphicx,color,amssymb}
\usepackage{amsmath}
\usepackage{siunitx}
\usepackage{tikz}
\usepackage{pgfplots}
\usepackage{caption,subcaption}
\usepackage{here}
\usepackage{url}
\usepackage{fullpage}
\usepackage{pgfplots}
\usepackage{caption}
\usepackage{graphicx}
\usepackage{subcaption}
\usepackage{float}
\usepackage[toc,page]{appendix}
\usepackage[utf8]{inputenc}
\usepackage{comment}
\usepackage{dsfont}
\usepackage{hyperref}
\usepackage{listings}
\usepackage{multicol}

\def\R{\mathbb{R}}
\def\C{\mathbb{C}}
\def\N{\mathbb{N}}

\def\Z{\mathbb{Z}}
\def\S{\mathcal{S}}

\def\two{\bar{\textbf{u}}_n}
\def\three{\dot{\textbf{u}}}

\def\torus{{\mathds{T}^d}}
\def\ltorus{{L^2(\torus;\mathbb{C})}}
\def\ltoruszero{{L^2_0(\torus;\mathbb{C})}}
\def\lsmall{\ell^2(\Z^d)}
\def\norml#1{\lVert #1\rVert_{\ltorus}}

\def\normh2s#1{\lVert #1\rVert_{H^{2s}}}

\def\frak#1{(-\Delta)^{#1}}
\def\roomhs#1{H^{#1}(\torus;\mathbb{C})}
\def\roomhszero#1{H^{#1}_0(\torus;\mathbb{C})}
\def\roomhshomo#1{\dot{H}^{#1}(\torus;\mathbb{C})}
\DeclareMathOperator{\dive}{div}
\def\makeheadbox{\relax}
\date{}

\begin{document}
	\title{Parameter learning and fractional differential operators: application in image regularization and decomposition}
	\author{Sören Bartels \and Nico Weber}
	\institute{
		Sören Bartels \at Department of Applied Mathematics, Mathematical Institute, University of Freiburg, \email{bartels@mathematik.uni-freiburg.de \and Nico Weber \at Department of Applied Mathematics, Mathematical Institute, University of Freiburg, \email{nico.weber@mathematik.uni-freiburg.de}}
	}
	\maketitle
	\begin{abstract}
		In this paper, we focus on learning optimal parameters for PDE-based image regularization and decomposition. First we learn the regularization parameter and the differential operator for gray-scale image denoising using the fractional Laplacian in combination with a bilevel optimization problem.
		In our setting the fractional Laplacian allows the use of Fourier transform, which enables the optimization of the denoising operator. We prove stable and explainable results as an advantage in comparison to other machine learning approaches. The numerical experiments correlate with our theoretical model setting and show a reduction of computing time in contrast to the ROF model. Second we introduce a new image decomposition model with the fractional Laplacian and the Riesz potential. We provide an explicit formula for the unique solution and the numerical experiments illustrate the efficiency.
	\end{abstract}
	\section{Introduction}
	In the last few years machine learning approaches have been established in image processing and computer vision. In contrast to this, variational regularization methods are used in image processing and computer vision since decades. Variational regularization techniques offer rigourous and comprehensible image analysis, which allows stable numerical results and error estimates. The certainty of giving explainable results is essential in a broad field of applications. Machine learning methods, on the other hand, are extremly powerful as they learn directly from datas for a specific task. The weak point of data-driven approaches is that they generally cannot offer stability or error bounds. In this paper we want to combine machine learning and variational regularization techniques. In particular, we learn optimal image regularizers and data fidelity parameters via a bilevel optimization approach making use of a training set.
	One of the central problems in image processing is denoising. Total variation image denoising is done with the so-called Rudin-Osher-Fatemi (ROF) model \cite{RudOshFat92}, which seeks a minimizer $u \in BV(\torus) \cap L^2(\torus)$ for
	\begin{align*}
	\hspace{1cm} E(u) = \big| Du \big|_{\torus} + \frac{\alpha }{2}\lVert g - u \rVert^2.
	\end{align*}
	The $d$-dimensional torus $\torus$ denotes the image domain, $\lVert \cdot \rVert$ is the norm in $L^2(\torus,\C)$, and $\alpha$ is a regularization parameter. The function $g: \torus \rightarrow \C$ represents the given image, which typically contains noise. A numerical difficulty of the ROF model is the non-differentiability of the total variation term. In the last years the use of differential operators, which involves fractional powers, were applied on many different kinds of problems, for example in \cite{frak1}. In image denoising as an alternative to the ROF model the total variation term is replaced by the fractional Laplacian.
	Fractional Laplacian denoising of an image is given by minimizing 
	\begin{align*}
	E(u) = \frac{1}{2}\lVert \frak{\frac{s}{2}}u \rVert^2_{\ltorus} + \frac{\alpha}{2}\lVert u-g \rVert^2_{\ltorus}.
	\end{align*}
	The application of the fractional Laplacian in image denoising has been done in \cite{AntBar17}.
	The results of this fractional model has a computing time, which  is a reduction by factors 10-100 in contrast to the well-known ROF model. Also in \cite{cguys} and \cite{disantil} the fractional Laplacian has been used in image denoising and got comparable results to the ROF model. Independent of the concrete choice of the model, the main issue in image regularization is the choice of particular regularization parameters, in our case the  two parameters $s$ and $\alpha$.  A well-known approach to compute suitable parameters is to define a bilevel optimization problem. A bilevel approach regarding a $TV_p$-image denoising model is studied in \cite{LiuSch19}. But the authors point out, that their scheme is numerically inefficient. To find the optimal regularization parameters they discretize the parameter interval and iterate over every grid point. An alternative approach using the fractional Laplacian is done in \cite{AntKha19}. There the authors learn the parameters via a so called Bilevel Optimization Neural Network. But as a disadvantage, which is typical for Neural Networks, no error estimates for the solutions are available. 
	\subsection{Contribution of this work}
	In this paper we obtain an image denoising model using the fractional Laplacian, which is on the one hand numerically fast to compute and on the other hand analytically understandable. We prove rigorous error estimates for the continous and the discrete solution.
	We consider the case of supervised learning. This means we are given noisy images $g$ and the corresponding noise-free image $u_d$. For simplicity we perform our model on a single pair $(g,u_d)$, but for multiple image pairs the results are a straightforward modification.
	The main idea of our model is to learn the optimal parameters $s$ and $\alpha$ on training data. The bilevel optimization problem is defined via
	\begin{align*}
	\hspace{1cm}&\min J(s,\alpha, u) = \frac{1}{2} \lVert u - u_d \rVert^2 + \varphi(s,\alpha)\\
	\hspace{1cm}&\text{s.t. } \frak{s} u + \alpha u = \alpha
	 g \text{ in } \torus \\
	&\text{with } (s,\alpha) \in W.
	\end{align*}
	The properties and the role of the function $\varphi$ will be discussed later.
	A theoretical analysis of this type of optimization problem in a more general setting is found in \cite{Sprekel}. 
	In the second part of the paper we apply fractional differential operators in image decomposition. 
	We consider a novel image decomposition model of the form
	\begin{align*}
	\hspace{1cm} I(u,v) &= \frac{1}{2} \lVert\frak{\frac{s_1}{2}}u\rVert^2  + \\
	 \hspace{1cm}&\frac{\alpha}{2} \lVert u + v - g\rVert^2 + \frac{\beta}{2}  \lVert R_{\frac{s_2}{2}}(v) \rVert ^2.
	\end{align*}  
	\begin{figure}
	\fboxsep=2mm
	\centering
	\begin{subfigure}{0.15\textwidth}
		\centering
		\mbox{
			\includegraphics[width=0.9\textwidth]{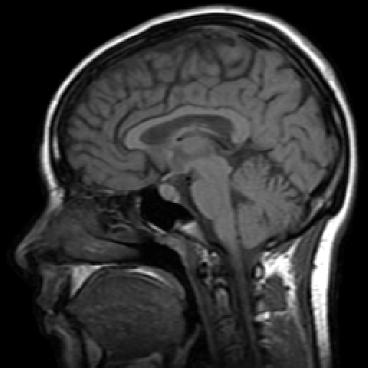}}
		\subcaption{}
	\end{subfigure}
	\begin{subfigure}{0.15\textwidth}
		\centering
		\mbox{
			\includegraphics[width=0.9\textwidth]{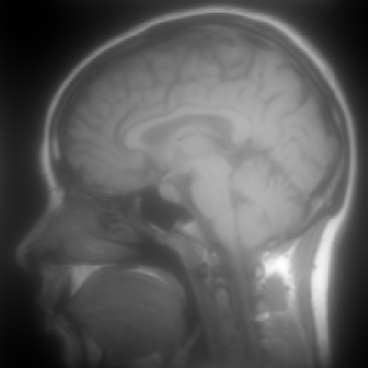}}
		\subcaption{}
	\end{subfigure}
	\begin{subfigure}{0.15\textwidth}
		\centering
		\mbox{
			\includegraphics[width=0.9\textwidth]{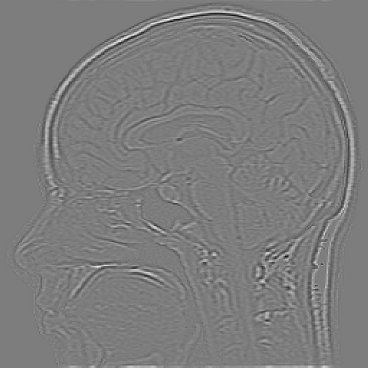}}
		\subcaption{}
	\end{subfigure}
	\caption{Decomposition of the original image (left) in structural component $u$ (middle) and textural component $v$ (right)}
	\label{decomposition}
	\end{figure}
	Typically the structural component $u$ contains the main components of the image, which is represented by the lower frequencies. The textural component $v$ contains the finer details like edges or high oscillations. These components are included in the high frequencies of an image. Therefore we introduce the Riesz potential, which captures the high frequencies, as the inverse of the fractional Laplacian. In Figure \ref{decomposition} we illustrate a decomposition into these components.
	This paper contains two main results. First we derive a rigourous error estimate for the bilevel problem in the case of spectral approximation. The numerical realization is easy to implement and in comparison to the ROF model much faster.
	Second we introduce a new image decomposition model, which has promising results in simultaneous image denoising and decomposition. To the best of our knowledge, this is the first work using fractional models in image decomposition.
	Moreover, we point out, that fractional differential operators are applicable in case of image denoising and decompositon. The use of fractional diffential operators in other image processing areas is a point of future research. 
	\subsection{Overview}
	This paper is organized as follows: In Section 2 we recall some facts about fractional Sobolev spaces and spectral approximation. The use of the fractional Laplacian in a bilevel image denoising approach is topic of Section 3. Afterwards in Section 4 a spectral approximation of this problem is introduced. In Section 5 we discuss the numerical experiments. An image decomposition approach is done in Section 6 with numerical experiments in Section 7. 
	\section{Fractional Sobolev spaces and spectral approximation}
	In this paragraph we collect some well known results concering fractional Sobolev spaces and spectral approximation. We follow \cite{AntBar17}.
	\subsection{Fractional Sobolev spaces}
	On the torus $\torus$ the Laplacian operator \\$-\Delta: \roomhszero{2} \rightarrow \ltoruszero$ is defined via
	\begin{align*}
	\hspace{1cm} - \Delta u = \frac{1}{(2\pi)^d} \sum\limits_{k \in \Z^d\setminus\{0\}}^{} |k|^{2}\hat{u}_k \varphi^k.
	\end{align*}
	with the functions $\varphi^k(x) = e^ {ik\cdot x}$ and $\hat{u}_k = \\(u,\varphi^k)_{L^2(\torus;\C)}$.
	In this setting the Laplacian operator is unbounded, non-negativ, self-adjoint and bijective.
	Therefore we can apply the theorem "Operators with compact inverses" (\cite{BoyFab13}, Theorem II-6.6). As a result we can characterize the domain of the Laplacian as
	\begin{align*}
	&D(-\Delta) = \roomhszero{2} = \\
	\hspace{1cm}&\{ u \in \ltoruszero ~\big| \sum\limits_{k \in \Z^d \setminus \{0\} }^{} |k|^4 |\hat{u}_k|^2 < \infty \}.
	\end{align*}
	Using this spectral decomposition we can define the fractional Sobolev spaces with $0\leq s \leq 1$ as 
	\begin{align*}
	&\roomhszero{s} :=\\
	\hspace{1cm}&\Big \{  u \in \ltoruszero ~\big | \sum\limits_{k \in  \Z^d \setminus \{0\}}^{} |k|^{4s} |\hat{u}_k|^2 < \infty \Big \}.
	\end{align*}
	The next theorem gives us analogous embeddings as in the case of general Sobolev spaces.
	\begin{theorem}[Rellich's theorem on $\torus$]\label{rellich}
		For all $0 \leq s < s'$ we have $\roomhszero{s'} \subset \roomhszero{s}$. \\Moreover, the canonical embedding is compact.
	\end{theorem}
	\begin{proof}
		\cite{BoyFab13} Proposition II.6.8. \qed
	\end{proof}
	The fractional Sobolev spaces are the natural setting to generalize the Laplacian operator on the torus.
	\begin{definition}[Fractional  Laplacian]
		For $s > 0$ we define the fractional Laplacian $\frak{s}$, applied on  $u \in \roomhszero{s}$, as
		\begin{align}
		\hspace{1cm}\frak{s} u = \frac{1}{(2\pi)^d} \sum\limits_{k \in \Z^d}^{} |k|^{2s}\hat{u}_k \varphi^k. \label{frakla}
		\end{align}
	\end{definition}
	\begin{figure}
	\fboxsep=2mm
	\centering
	\begin{subfigure}{0.22\textwidth}
		\centering
		\mbox{
			\includegraphics[width=0.9\textwidth]{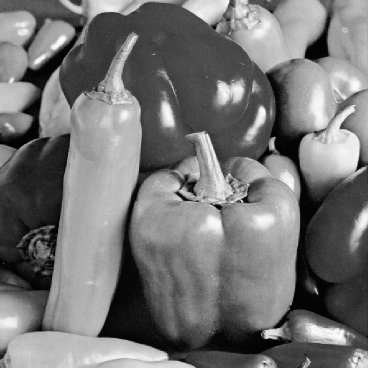}}
		\subcaption{$\frak{0} u_d = id~ u_d $}
	\end{subfigure}
	\vspace{0.2cm}
	\begin{subfigure}{0.22\textwidth}
		\centering
		\mbox{
			\includegraphics[width=0.9\textwidth]{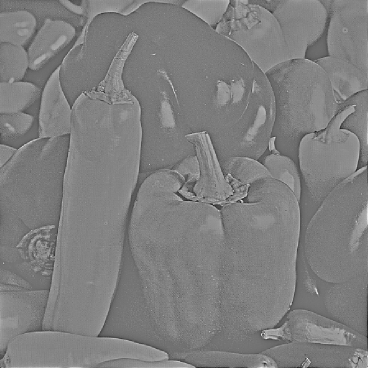}}
		\subcaption{$ \frak{\frac{1}{3}}u_d$}
	\end{subfigure}
	\begin{subfigure}{0.22\textwidth}
		\centering
		\mbox{
			\includegraphics[width=0.9\textwidth]{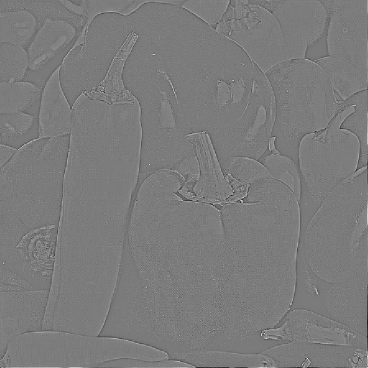}}
		\subcaption{$ \frak{\frac{2}{3}}u_d$}
	\end{subfigure}
	\begin{subfigure}{0.22\textwidth}
		\centering
		\mbox{
			\includegraphics[width=0.9\textwidth]{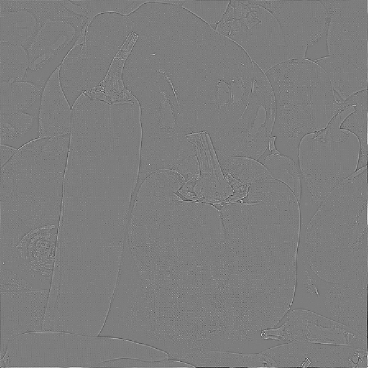}}
		\subcaption{$ \frak{1} u_d = -\Delta u_d$}
	\end{subfigure}
	\caption{Fractional Laplacian (\ref{frakla}) of image $u_d$. A higher exponent $s$ in (\ref{frakla}) emphasizes stronger the high oscillations.}
	\end{figure}
	The fractional Laplacian is a linear operator for the argument $u$, but non-linear in the parameter $s$.
	Moreover we  can define fractional Sobolev spaces for negative $s$.
	The domain $D(\frak{s})$ of the fractional Laplacian  with $s <0$ must fullfill the property of being a super set from $\ltoruszero$; so we modify the norm for negative $s$ with
	\begin{align*}
	\hspace{1cm}\lVert u \rVert^2_{D(\frak{s})} := \sum\limits_{k \in \Z^d \setminus \{0\}}^{} |k|^{4s} |\hat{u}_k|^2.
	\end{align*}
	
	\begin{definition}[Negative Sobolev spaces]\\
		Let $s < 0$. We refer $\roomhshomo{s}$ as the negative Sobolev spaces defined as the completition of $\ltoruszero$ regarding the norm $\lVert \cdot \rVert_{D(\frak{s})}$.
	\end{definition}
	It can be shown that $\roomhshomo{s}$ is a Hilbert space with the scalar product
	\begin{align*}
	\hspace{1cm}(u,v)_{\roomhshomo{s}} = \sum\limits_{k \in \Z^d \setminus \{0\}}^{} |k|^{2s} \hat{u}_k \overline{\hat{v}_k}.
	\end{align*}
	We are now in a position to define the Riesz potential, which is the negative analogue to the fractional Laplacian.
	\begin{definition}[Riesz potential]
		For $s \leq 0 $ we define the Riesz potential of a function \\$ u \in \roomhshomo{s}$ as the inverse of the fractional Laplacian with
		\begin{align}
		\hspace{1cm}R_s (u) = \frac{1}{(2\pi)^d}\sum\limits_{k \in \Z^d \setminus \{0\}}^{} |k|^{2s} \hat{u}_k \varphi^k . \label{riesz}
		\end{align}
	\end{definition}
	\begin{figure}
		\fboxsep=2mm
		\centering
		\begin{subfigure}{0.22\textwidth}
			\centering
			\mbox{
				\includegraphics[width=0.9\textwidth]{comp1.png}}
			\subcaption{original image $u_d$}
		\end{subfigure}
		\vspace{0.2cm}
		\begin{subfigure}{0.22\textwidth}
			\centering
			\mbox{
				\includegraphics[width=0.9\textwidth]{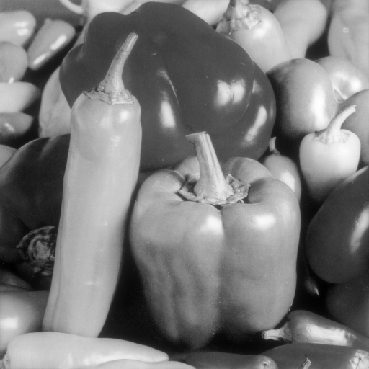}}
			\subcaption{$R_{-0.1}(u_d)$}
		\end{subfigure}
		\begin{subfigure}{0.22\textwidth}
			\centering
			\mbox{
				\includegraphics[width=0.9\textwidth]{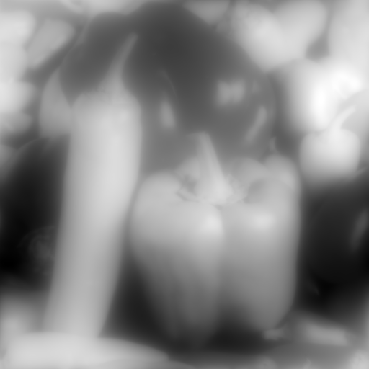}}
			\subcaption{$R_{-0.5}(u_d)$}
		\end{subfigure}
		\begin{subfigure}{0.22\textwidth}
			\centering
			\mbox{
				\includegraphics[width=0.9\textwidth]{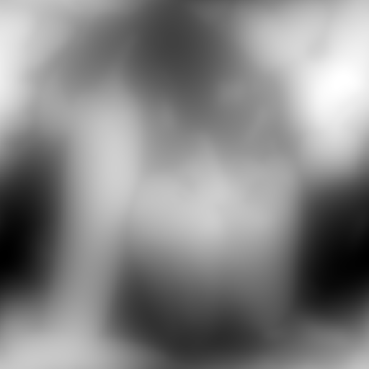}}
			\subcaption{$R_{-0.9}(u_d)$}
		\end{subfigure}
		\caption{Riesz potential (\ref{riesz}) of image $u_d$. A lower exponent $s$ in (\ref{riesz}) smooths the image $u_d$.}
	\end{figure}
	\subsection{Spectral approximation}
	For short summaries about spectral approximation, we refer to \cite{AntBar17} and \cite{spec}.
	The space of trigonometric polynomials  on the torus $\torus$ is defined via 
		\begin{align*}
		  \hspace{1cm}\mathcal{T}_n &= 
		  \Big\{v_n \in C(\torus,\C)~\\
		  \hspace{1cm}&\big |~ v_n(x) =  \sum\limits_{k \in \Z^d_n}^{} c_k\varphi^k(x), c_k \in  \C\Big\},
		\end{align*}
		with $\Z^d_n := \{ k \in \Z^d ~~\big| \frac{-n}{2} \leq k_i \leq \frac{n}{2}-1 ~\forall i\}$.
	The functions $\varphi^k(x) = e^{ik\cdot x}$ define an orthogonal basis in $\mathcal{T}_n$ regarding the $L^2$-scalar product.\\
	With  $v$  we associate a  grid function $\textbf{V} = (v_j | ~j \in \N^d_n)$, which is represented by 
	\begin{align*}
	&\hspace{1cm}v_j = v(x_j) \text{ , } x_j = \frac{2\pi}{n} (j_1,...,j_d),\\& \hspace{1cm} j \in \N^d_n = \{~ j \in \N^d ~|~ 0 \leq j_i \leq n-1~\forall j\}.
	\end{align*}
	
	The discrete scalar product of two grid functions $\textbf{V} = (v_j ~|~ j \in \N^d_n)$ and $\textbf{W} = (w_j ~|~ j \in \N^d_n)$  is given by
		\begin{align*}
		\hspace{1cm}(\textbf{V},\textbf{W})_n := \frac{(2\pi)^d}{n^d} \sum_{j \in \N^d_n}^{} v_j \bar{w}_j.
		\end{align*}
	The induced norm is denoted by $\Vert \cdot \Vert_n$.

	\begin{definition}[Fourier transform]
		For a given grid function $ \textbf{V} = ( v_j ~|~ j \in \N^d_n)$  the discrete Fourier transform is defined as the coefficient vector $\tilde{\textbf{V}} = (\tilde{v}_k |~ k \in \Z^d_n)$ with
		\begin{align*}
		\hspace{1cm}\tilde{v}_k = (V,\Phi^k)_n.
		\end{align*}
		The family 
		\begin{align*}
		\hspace{1cm}\Phi^k = ( e^{ik\cdot x_j} ~|~ j \in \N^d_n) ~,~ x_j = \frac{2\pi}{n}(j_1,...,j_d)
		\end{align*}
		defines an orthogonal basis in the space of grid functions regarding  $(\cdot,\cdot)_n$.
	\end{definition}
	To approximate a function in the fractional \\Sobolev space $\roomhszero{s}$, we consider suitable approximation in the trigonometric space $\mathcal{T}_n$.
	\begin{definition}[Orthogonal projection]\\
		The projection $P_n: \ltorus \rightarrow \mathcal{T}_n $ fulfills for all $v  \in \ltorus$ the property
		\begin{align*}
		\hspace{1cm}(P_n v, w_n) = (v, w_n)~~~ \forall w_n \in \mathcal{T}_n.
		\end{align*}
	\end{definition}
	Using the orthogonality of the basis functions we get
	\begin{align*}
	\hspace{1cm}P_n v = \sum\limits_{ k \in \Z^d_n}^{} \hat{v}_k \varphi^k.
	\end{align*} 
	
	The discrete Fourier transformation allows us to define a  suitable trigonometric interpolation.
	\begin{definition}[Trigonometric interpolant]
		Given $v \in C(\torus;\C)$ and discrete Fourier coefficients $\tilde{\textbf{V}} = (\tilde{v}_k ~|~ k \in \Z^d_n)$, the trigonometric interpolant $I_n v \in \mathcal{T}_n$ of $v$ is defined via
		\begin{align*}
		 \hspace{1cm}I_n v = \frac{1}{(2\pi)^d} \sum_{k \in \Z^d_n}^{} \tilde{v}_k \varphi^k.
		\end{align*}
	\end{definition}
	The following error estimate can be found in \cite{AntBar17}. 
	The norm $\lVert \cdot \rVert_{D(\frak{s})}$ is denoted by $|\cdot|_s$.
	\begin{lemma}[Projection error]\label{proj}
		For $\gamma_1,\gamma_2 \in \R$ with $\gamma_1 \leq \gamma_2$ and $v \in \roomhszero{\gamma_2}$ we obtain
		\begin{align*}
		\hspace{1cm}| v - P_n v|_{\gamma_1} \leq \Big(\frac{n}{2}\Big)^{-(\gamma_2-\gamma_1)}|v|_{\gamma_2}.
		\end{align*}
	\end{lemma}
	
	\begin{lemma}[Interpolation error]\label{inter}
		If $\gamma_2 > \frac{d}{2}, 0 \leq \gamma_1 \leq \gamma_2$ and $v \in \roomhszero{\gamma_2}$ we obtain 
		\begin{equation*}
		\hspace{1cm}|v - I_n v |_{\gamma_1} \leq c_{d,\gamma_1,\gamma_2}\Big(\frac{n}{2}\Big)^{-(\gamma_2-\gamma_1)} |v|_{\gamma_2}
		\end{equation*}
		with a constant $c_{d,\gamma_1,\gamma_2} > 0$ independent of $v$ and $n$.
	\end{lemma}
	\section{Application of fractional Laplacian in bilevel image optimization}
	As mentioned in the introduction for a given noisy image $g \in \ltorus$ we want to solve 
	\begin{align}
	\hspace{0.5cm} \min ~E(u) = \frac{1}{2}\lVert \frak{\frac{s}{2}}u \rVert^2 + \frac{\alpha}{2}\lVert u-g \rVert^2 \label{funcbartels}
	\end{align} 
	with $ 0 < s < 1$ and $u \in \roomhszero{s} \cap \ltoruszero$. The main idea of this model is to supress the high frequencies, which typically contain noise. In the following we always assume, that the mean value of $g$ is zero. Otherwise, we replace $g$ with
	\begin{align*}
	\hspace{1cm} \tilde{g}(x) := g(x) - \frac{1}{\big|\torus\big| } \int_{\torus}^{} g(x)~dx.
	\end{align*}
	and add the mean value to the solution. 
	As a result we can minimize in the function space $\roomhszero{s}$. Existence and uniqueness of the solution $u$ considered in the frequency space yields 
	\begin{equation}\label{u}
	\hspace{1cm} u = \frac{1}{(2\pi)^d} \sum\limits_{k \in \Z^d \setminus \{0\}}^{} \frac{\alpha}{|k|^{2s}+\alpha}\,\hat{g}_k\varphi^k.
	\end{equation}
	\begin{figure}
	\fboxsep=2mm
	\centering
	\begin{subfigure}{0.22\textwidth}
	\centering
	\mbox{
		\includegraphics[width=0.9\textwidth]{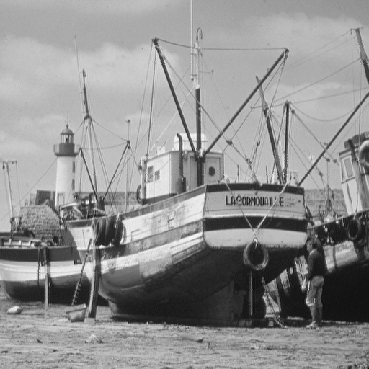}}
	\subcaption{original image $u_d$}
	\end{subfigure}
	\vspace{0.2cm}
	\begin{subfigure}{0.22\textwidth}
	\centering
	\mbox{
		\includegraphics[width=0.9\textwidth]{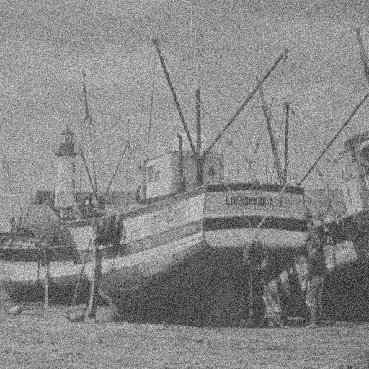}}
	\subcaption{noisy image $g$}
	\end{subfigure}
	\begin{subfigure}{0.22\textwidth}
	\centering
	\mbox{
		\includegraphics[width=0.9\textwidth]{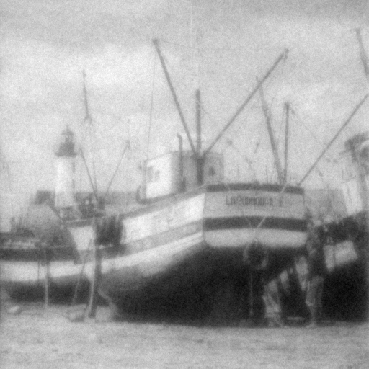}}
	\subcaption{$s=0.4,~\alpha=5$}
	\end{subfigure}
	\begin{subfigure}{0.22\textwidth}
	\centering
	\mbox{
		\includegraphics[width=0.9\textwidth]{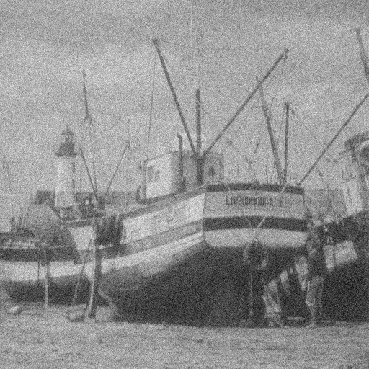}}
	\subcaption{$s=0.1,~\alpha=1$}
	\end{subfigure}
	\caption{The denoising via solving (\ref{funcbartels}) strongly depend on the choice of the parameters $s$ and $\alpha$.}
	\label{influence}
	\end{figure}
	A key point in image denoising is the choice of the parameters, in this case $s$ and $\alpha$. Normally this is done manually, which is time-consuming. As seen in Figure \ref{influence} the results of denoising an image highly depends on the right choice of the regularization parameters. To overcome this problem we use a bilevel optimization approach. We assume, that we have the original image $u_d \in \ltorus$ and the noisy image $g$. So we define the optimization problem 
	\begin{align}
	\min_{s,\alpha,u}&~J(s,\alpha,u) = \frac{1}{2} \norml{u -u_d} ^2 + \varphi(\alpha,s)\nonumber\\
	&\text{s.t. } \frak{s}u + \alpha u = \alpha g \label{func1}\\
	&\text{with }(s, \alpha) \in W := [s_0,s_1]  \times [\alpha_0,\alpha_1]\nonumber.
	\end{align}
	The function $\varphi$ must fulfill following assumptions.
	\begin{definition}
		The function $\varphi \in C^2(W^{\mathrm{o}},\R)$ is non negative, convex and has following properties:
		\begin{align}
		&\hspace{1cm}\lim\limits_{s \rightarrow s_0} \varphi(s,\alpha) = \lim\limits_{\alpha \rightarrow \alpha_0} \varphi(s,\alpha) = \infty \nonumber \\
		&\hspace{1cm} = \lim\limits_{s \rightarrow s_1} \varphi(s,\alpha) = \lim\limits_{\alpha \rightarrow \alpha_1} \varphi(s,\alpha). \label{steil}
		\end{align} 
	\end{definition}
	Parsevals's identity and the isometry properties between $\ltorus$ and $\lsmall$ imply the equivalent representation of (\ref{func1}) with 
	\begin{align}
	J(s,\alpha,u) = \frac{1}{2(2\pi)^d} \lVert \widehat{\textbf{U}} - \widehat{\textbf{U}}_d \rVert_{l^2}^2   + \varphi(\alpha,s) \label{funfourier1}\\
	\text{s.t. }((\alpha+ |k|^{2s})\hat{u}_k - \alpha \hat{g_k})_{k \in \mathbb{Z}^d\setminus\{0\}}  = 0 \label{funfourier2}\\
	\text{with }(s, \alpha) \in W := [s_0,s_1]  \times [\alpha_0,\alpha_1] \nonumber.
	\end{align}
	The expressions $\widehat{\textbf{U}}$ and $\widehat{\textbf{U}}_d$ denote the Fourier coefficient vector of $u$ and $u_d$.
	Solving (\ref{funfourier2}) according to $\hat{u}_k$ yields a solution operator $\mathcal{S}$ defined via
	\begin{align}
	\mathcal{S}: [s_0,s_1]\times[\alpha_0,\alpha_1] \rightarrow \lsmall \nonumber\\
	(s,\alpha) \mapsto \mathcal{S}(s,\alpha) = \widehat{\textbf{U}} =  \big(\frac{\alpha}{\alpha + |k|^{2s}} \hat{g_k}\big)_{k \in \mathbb{Z}^d}. \label{operator2}
	\end{align}
	To analyze the optimization problem (\ref{funfourier1})-(\ref{funfourier2}), we need to characterize the partial derivatives of the solution operator $\mathcal{S}$.
	\begin{theorem}\label{diffbar}
		The operator $\mathcal{S}$  as an operator from $ W \rightarrow \ltorus \simeq \lsmall$ is twice Fr\'{e}chet-differentiable  and in particular continous.
		For $h_1, h_2 \in \R$  the partial derivatives up to order two are represented by
		\begin{align*}
		&\hspace{1cm}\partial_s \mathcal{S}(s,\alpha)[h_1] = h_1\partial_s\mathcal{S}(s,\alpha),\\
		&\hspace{1cm}\partial^2_{s} \mathcal{S}(s,\alpha)[h_1,h_2] = h_1h_2 \partial_{s^2} \mathcal{S}(s,\alpha),\\
		&\hspace{1cm}\partial_\alpha \mathcal{S}(s,\alpha)[h_1] = h_1\partial_\alpha\mathcal{S}(s,\alpha),\\
		&\hspace{1cm}\partial^2_{\alpha} \mathcal{S}(s,\alpha)[h_1,h_2] = h_1h_2 \partial_{\alpha^2} \mathcal{S}(s,\alpha),\\
		&\hspace{1cm}\partial_{\alpha, s } \mathcal{S}(s,\alpha)[h_1,h_2] = h_1h_2 \partial_{\alpha,s} \mathcal{S}(s,\alpha),
		\end{align*}
		with 
		\begin{align*}
		&\partial_s \mathcal{S}(s,\alpha) = \sum\limits_{k \in \Z^d \setminus \{0\}}^{} 	 \frac{2\alpha|k|^{2s}ln(|k|)}{(|k|^{2s} + \alpha)^2} \hat{g}_k,\\
		&\partial^2_{s} \mathcal{S}(s,\alpha) = \sum\limits_{k \in \Z^d \setminus \{ 0\}}^{} \frac{4\alpha|k|^{2s}ln(|k|)^2(|k|^{2s}-\alpha)}{(\alpha + |k|^{2s})^3} \hat{g}_k, \\
		&\partial_\alpha \mathcal{S}(s,\alpha) = \sum\limits_{k \in \Z^d \setminus \{0\}}^{} \frac{|k|^{2s}}{(|k|^{2s}+ \alpha)^2}\hat{g}_k,\\ &\partial^2_{\alpha} \mathcal{S}(s,\alpha) = \sum\limits_{k \in \Z^d \setminus \{0\}}^{}  \frac{-2|k|^{2s}}{(|k|^{2s} + \alpha )^3} \hat{g}_k,\\
		&\partial_{\alpha, s} \mathcal{S}(s,\alpha) = \sum\limits_{k \in \Z^d \setminus \{0\}}^{} \!
		\frac{-2|k|^{2s}ln(|k|)(|k|^{2s} - \alpha)}{(|k|^{2s} + \alpha)^3}\hat{g}_k.
		\end{align*}
	\end{theorem}
	\begin{proof}
	The proof requires straightforward calculations, but follows mainly the proof structure of Section 3.1 in \cite{AntOtaSal18}. A more general case is also found in \cite{Sprekel}. In our case we focus on the function 
	$E_k: [s_0,s_1]\times[\alpha_0,\alpha_1] \rightarrow \R$ for $k \in \Z^d$ with 
	\begin{align*}
	~~~~~~E_k(s,\alpha) := \frac{\alpha}{\alpha + |k|^{2s}}.
	\end{align*}
	The remaining part of the proof is a straightforward modification of the cited paper.\qed
	\end{proof}
	Using the explicit representation of the partial derivatives we obtain upper bounds for these.
	 \begin{lemma}[Boundedness of partial \\derivatives]\label{bepartial}
		Let $\varepsilon > 0$ and $i=1,2,3$. Then we have the following estimates for the partial derivatives: 
		\begin{align*}
		&\hspace{1cm}\Vert \partial^i_\alpha \mathcal{S}(s,\alpha)\Vert_{\ell^2} \leq \sum\limits_{k \in \mathbb{Z}^d\setminus\{0\}}^{} \frac{i!}{|k|^{2si}} |\hat{g_k}|,\\
		&\hspace{1cm}\Vert \partial^i_s \mathcal{S}(s,\alpha)\Vert_{\ell^2} \leq \sum\limits_{k \in \mathbb{Z}^d\setminus\{0\}}^{} \frac{M_{\alpha_1,i}}{\varepsilon^i |k|^{2s- i\varepsilon}} |\hat{g}_k|,	 \\
		&\hspace{1cm}\Vert \partial_{s,\alpha} \mathcal{S}(s,\alpha)\Vert_{\ell^2} \leq \sum\limits_{k \in \mathbb{Z}^d\setminus\{0\}}^{} \frac{M_{\alpha_1}}{\varepsilon |k|^{2s- \varepsilon}} |\hat{g}_k|	.	
		\end{align*}		
	\end{lemma}
	\begin{proof}
		We obtain these results by using Theorem 2 and straightforward calculations.\qed
	\end{proof}
	If $g \in \ltorus$ and $\varepsilon > 0$, we deduce from Lemma ~\ref{bepartial}
	\begin{align*}
	&\hspace{1cm}\partial^i_s \S(s,\alpha) \in \roomhszero{2s - i\varepsilon}, \\ &\hspace{1cm}\partial^i_\alpha \S(s,\alpha) \in \roomhszero{2si}
	\end{align*}
	for $i=1,2,3$. \\
	The introduction of the solution operator  $\mathcal{S}$ allows us to consider a reduced version of (\ref{funfourier1})-(\ref{funfourier2}). The reduced problem  $j: \R^2 \mapsto \R$ of (\ref{funfourier1}) and (\ref{funfourier2}) is given by minimizing
	\begin{align}
	\hspace{0.5cm} j(s,\alpha) &:=  J(s,\alpha, \mathcal{S}(s,\alpha)) \label{funcred}\\
	&= \frac{1}{2(2\pi)^d} \lVert \mathcal{S}(s,\alpha) - \widehat{\textbf{U}}_d\rVert_{\ell^2}^2 + \varphi(s,\alpha) \nonumber\\ 
	&\text{with }(s, \alpha) \in W. \nonumber
	\end{align}
	As an optimal triple we denote a minimizer of $j(s,\alpha)$ together with $u = \mathcal{S}(s,\alpha)$.
	\begin{definition}[Optimal triple]
		The triple \\$(\bar{s}, \bar{\alpha}, \mathcal{S}(\bar{s}, \bar{\alpha})) \in W^{\mathrm{o}} \times \roomhszero{\overline{2s}}$ is called optimal for the problem (\ref{func1}) if
		\begin{align*}
		~~~~~~j(\bar{s}, \bar{\alpha}) \leq j(s,\alpha)
		\end{align*}
		for all $(s,\alpha, \mathcal{S}(s, \alpha)) \in W^{\mathrm{o}} \times \roomhszero{2s}$.
	\end{definition}
	
	The existence of an optimal triple can easily be shown.
	\begin{theorem}[Existence of a solution]\label{existence}
		There exists an optimal triple \\$(\bar{s}, \bar{\alpha}, \mathcal{S}(\bar{s}, \bar{\alpha})) \in  W^{\mathrm{o}}\times \roomhs{\overline{2s}}$ for problem (\ref{func1}).
	\end{theorem}
	\begin{proof}
		The proof follows \cite{AntOtaSal18}. \\
		Let $W_{\ell} = [s_{\ell_1},s_{\ell_2}]\times[\alpha_{\ell_1},\alpha_{\ell_2}]$ be a sequence of closed intervals with $W_{\ell} \subset W_{\ell+1} \subset W^{\mathrm{o}}$ and \\
		$\bigcup \limits_{\ell \in \N} W_{\ell} = W$. The continuity of $j$ guarantees the existence of
		\begin{align*}
		\hspace{1cm}(s_{\ell},\alpha_{\ell}) = \underset{(s,\alpha) \in W_{\ell}}{\arg \,min} ~~j(s,\alpha).
		\end{align*}
		Because of the construction of the intervals we have 
		\begin{align*}
		\hspace{1cm}j(s_m, \alpha_m) \leq j(s_{\ell}, \alpha_{\ell})
		\end{align*}
		for all $m  \geq \ell$.\\
		As a result we get a convergent subsequence \\$\{(s_{\ell}, \alpha_{\ell})\}_{\ell \in \N}$ with $(s_{\ell}, \alpha_{\ell}) \rightarrow (\bar{s}, \bar{\alpha}) \in W$. \\
		The property
		\begin{equation}\label{min}
		\hspace{1cm}j(\bar{s}, \bar{\alpha}) \leq j(s_{\ell}, \alpha_{\ell}) 
		\end{equation}
		for all $\ell \in \N$ and $j(s,\alpha) \geq \varphi(s,\alpha)$ together with assumption (\ref{steil}) yields $ (\bar{s}, \bar{\alpha}) \in  W^{\mathrm{o}}$ and the minimizer property. \\
		Due the continuity and the closed range of $\mathcal{S}$ there exists a 
		$\bar{\textbf{u}}$ in the image of $\mathcal{S}$ with
		\begin{align*}
		\hspace{1cm }\mathcal{S}(s_{\ell},\alpha_{\ell}) \rightarrow \bar{\textbf{u}} = (\frac{\alpha}{\alpha+ |k|^{\bar{2s}}}\hat{g}_k)_{k \in \Z^d\setminus \{0\}}
		\end{align*}
		for $\ell \rightarrow \infty$.\\
		Therefore we obtain $\mathcal{S}(\bar{s}, \bar{\alpha}) \in \roomhszero{\overline{2s}}$.\qed
	\end{proof}
	The reduced problem is essentially a restricted optimization problem in 
	$\R^2$, so we can easily provide first order necessary and second order sufficient optimality conditions.
	\begin{theorem}[Optimality conditions]\label{opt}\\
		Let $(\bar{s},\bar{\alpha}, \mathcal{S}(\bar{s}, \bar{\alpha}))$ be an optimal triple for problem (\ref{func1}). For shorter notation we define $\bar{\textbf{u}} := \mathcal{S}(\bar{s}, \bar{\alpha})$. Then the first order optimality conditions must be valid, this implies
		\begin{align*}
		&\hspace{0.5cm}\frac{1}{2(2\pi)^d}\big((\bar{\textbf{u}} - \widehat{\textbf{U}}_d, \partial_s \bar{\textbf{u}})_{l^2} + ( \partial_s \bar{\textbf{u}},\bar{\textbf{u}} - \widehat{\textbf{U}}_d)_{l^2}\big) \\
		&\hspace{0.5cm}+ \partial_s \varphi(\bar{s},\bar{\alpha}) = 0,\\
		&\hspace{0.5cm}\frac{1}{2(2\pi)^d}\big((\bar{\textbf{u}} - \widehat{\textbf{U}}_d, \partial_\alpha \bar{\textbf{u}})_{l^2} +( \partial_\alpha \bar{\textbf{u}},\bar{\textbf{u}} - \widehat{\textbf{U}}_d)_{l^2}\big)\\
		&\hspace{0.5cm}+ \partial_\alpha \varphi(\bar{s},\bar{\alpha}) = 0.
		\end{align*}
		If there exists a pair $(\bar{s}, \bar{\alpha})$ with $\bar{\textbf{u}} = \mathcal{S}(\bar{s}, \bar{\alpha})$, which fulfills the first order optimality conditions and the Hessian matrix $\textbf{A} \in \R^{2 \times 2}$ is positive definit, i.e.
		\begin{align*}
		\hspace{1cm}\det \begin{pmatrix}
		a_{1\,1} & a_{1\,2} \\
		a_{2\,1} & a_{2\,2}
		\end{pmatrix} > 0,
		\end{align*}
		then $(\bar{s}, \bar{\alpha}, \bar{\textbf{u}})$  is an optimal triple for (\ref{func1}).
	\end{theorem}
	To get the uniqueness of the optimal tripel, we need a stronger assumption on the function $\varphi$.
	\begin{definition}[Strong convexity, \cite{BoyVan04}]
		Let $W$ be a convex set. A two-times differentiable function $\varphi: W \subset \R^n \rightarrow \R$ is strongly convex, if there exists a constant $\theta > 0$ with 
		\begin{equation*}
		\hspace{1cm}\nabla^2\varphi(x) \succeq \theta \textbf{I}
		\end{equation*}
		for all $x \in W$ , i.e. the matrix $\nabla^2\varphi(x) - \theta I$ is positive definit.
	\end{definition}
	
		The definiton of strong convexity implies (cf. \cite{Nes04})
		\begin{equation} \label{strong1}
		\hspace{1cm}(\nabla\varphi(x)-\nabla\varphi(y))^T \big( x - y \big)\geq \theta \big|
		x - y\big| 
		\end{equation}
		for all $x,y \in W$
		and 
		\begin{equation} \label{strong2}
		\hspace{1cm}\varphi(x)  \geq \varphi(\bar{x}) + \frac{1}{2}\theta | (x - y) |^2 ~~~\forall y \in W.
		\end{equation}
		in the case $\nabla \varphi(\bar{x}) = 0$ for a  $\bar{x} \in W $.	
	In the following we always assume, that the function $\varphi$ is strongly convex. 
	\begin{lemma}\label{posdef}
		If $ \lVert g \rVert$ and
		$\lVert u_d \rVert$ are suffienctly small, then there exists a constant $\kappa > 0$, such that
		\begin{equation}\label{secondderivative}
		\hspace{1cm}\nabla^2 j(s,\alpha) \succeq \kappa \textbf{I}
		\end{equation}
		for all $(s,\alpha) \in W$.
	\end{lemma}
	
	\begin{proof}
		Decompose the Hessian matrix $\textbf{A}$ from  (\ref{opt}) in four parts with
		\begin{equation*}
		x^T\textbf{A}x = x^T\textbf{M}_1x + x^T\textbf{M}_2x + x^T\textbf{M}_3x + x^T\textbf{M}_4x
		\end{equation*}
		for arbitrary $x \in \R^2$.
		The matrix $\textbf{M}_1$ 
		\begin{align*}
		\hspace{1cm}\textbf{M}_1 & = \frac{1}{2(2\pi)^d}\left( \begin{array}{rr}
		m^1_{11}&m^1_{12}\\
		m^1_{21}&m^1_{22}
		\end{array}\right)
		\end{align*}
		has the values
		\begin{align*}
		m^1_{11} = (\mathcal{S}(s,\alpha) - \widehat{\textbf{U}}_d,\partial^2_{s}\mathcal{S}(s,\alpha))_{\ell^2},\\
		m^1_{12} = m^1_{21} = (\mathcal{S}(s,\alpha) - \widehat{\textbf{U}}_d,\partial^2_{s,\alpha}\mathcal{S}(s,\alpha))_{\ell^2},\\
		m^1_{22} = (\mathcal{S}(s,\alpha) - \widehat{\textbf{U}}_d,\partial^2_{\alpha}\mathcal{S}(s,\alpha))_{\ell^2}.
		\end{align*}
		The matrix $\textbf{M}_2$ is defined analogously to $\textbf{M}_1$, only the arguments in the $\ell^2$-scalar-product are interchanged. Therefore the following arguments for matrix $\textbf{M}_1$ are also valid for matrix $\textbf{M}_2$.
		We obtain 
		\begin{align*}
		\hspace{1cm}x^T \textbf{M}_1 x = m^1_{11}x^2_1 + 2m^1_{12}m^1_{21}x_1x_2 + m^1_{22}x^2_2 \\\hspace{1cm}\geq m^1_{11}x^2_1 -  2|m^1_{12}||m^1_{21}||x_1 x_2| + m^1_{22}x^2_2
		\end{align*}
		for arbitrary $x \in \R^2$. The Cauchy-Schwarz and Young inequality imply
		\begin{align*}
		\hspace{1cm}x^TM_1x \geq (-|m^1_{11}| - |m^1_{12}||m^1_{21}|)x^2_1\\
		\hspace{1cm}+~(-|m^1_{22}| - |m^1_{12}||m^1_{21}|)x^2_2.
		\end{align*}
		Together with Lemma \ref{bepartial} we have
		\begin{align*} 	
		\hspace{1cm}|m^1_{ij}| \leq \frac{C_{ij}}{s_0}( \lVert g \rVert + \lVert u_d \rVert)\lVert g \rVert.
		\end{align*}
		The matrix $\textbf{M}_3$ is the Gramian matrix between  the vectors $\partial_{s}\mathcal{S}(s,\alpha)$ and $\partial_{\alpha}\mathcal{S}(s,\alpha)$, which implies $x^T\textbf{M}_2x \geq 0$ for all $x \in \R^2$.\\
		The strong convexity of the function $\varphi$ yields for the matrix $\textbf{M}_4$, which is defined as 
		\begin{align*}
		\hspace{0.5cm}\textbf{M}_4 = \left( \begin{array}{rr}
		\partial^2_{s} \varphi(s,\alpha)&\partial^2_{s,\alpha} \varphi(s,\alpha)\\
		\partial^2_{s,\alpha} \varphi(s,\alpha)&\partial^2_{\alpha} \varphi(s,\alpha)
		\end{array}\right),
		\end{align*}
		the result 
		\begin{equation*}
		\hspace{1cm }x^T\textbf{M}_4x \geq \theta |x|^2.
		\end{equation*}
		Combining all arguments we obtain
		\begin{align*}
		\hspace{1cm}x^T\textbf{A}x \geq -C(\lVert u_d \rVert, \lVert g \rVert, \frac{1}{s_0})|x|^2 + \theta |x|^2.
		\end{align*}
		Therefore we have for suffienctly small $\lVert u_d \rVert$ and $\lVert g \rVert$, that the constant $C(\lVert u_d \rVert, \lVert g \rVert, \frac{1}{s_0})$ is smaller than $\theta$. This implies the positiv definitness of the matrix $\textbf{A}$.\qed
	\end{proof}
		If the  prerequisites of Lemma 4 are fulfilled, we obtain for an optimal triple $(\bar{s},\bar{\alpha}, \mathcal{S}(\bar{s}, \bar{\alpha}))$  
		\begin{align*}
		&\hspace{1cm}(\nabla j(s,\alpha)- \nabla j(\bar{s},\bar{\alpha}))^T\Big(\binom{s}{\alpha}- \binom{\bar{s}}{\bar{\alpha}}\Big) \\
		&\hspace{1cm}\geq \kappa \Big|\binom{s}{\alpha}- \binom{\bar{s}}{\bar{\alpha}}\Big|^2~~~
		\end{align*}
		and 
		\begin{align*}
		\hspace{1cm}j(s,\alpha) \geq j(\bar{s},\bar{\alpha}) + \frac{\kappa}{2} \Big|\binom{\bar{s}}{\bar{\alpha}}- \binom{s}{\alpha}\Big|^2~~~
		\end{align*}
		for all  $(s,\alpha) \in W$. The quadratic growth conditon implies the uniqueness of the optimal triple.
	\section{Spectral approximation and convergence analysis}
	For given $g_n \in \mathcal{T}_n$ we have the discrete problem
	\begin{align*}
	~~~~~~\frak{s}u_n + \alpha u_n = \alpha g_n ~~ \text{in} ~\mathcal{T}_n.
	\end{align*}
	Define the discrete solution operator $\mathcal{S}_n$ via
	\begin{align*}
	~~~~~\mathcal{S}_n: [s_0,s_1]\times[\alpha_0,\alpha_1] \rightarrow \mathcal{T}_n, \\
	~~~~~(s,\alpha) \mapsto (\frac{\alpha}{\alpha + |k|^{2s}}\hat{{g}}_{n,k})_{k \in \Z^d_n}.
	\end{align*}
	The following Lipschitz continuity holds.
	\begin{lemma}[Stability]\label{stabi}
		Let $\textbf{u} = \mathcal{S}(s, \alpha)$ be the continous and $\textbf{u}_n = \mathcal{S}_n(s, \alpha)$ the discrete solution for $(s,\alpha) \in W$.
		Then we have the following error estimate between these both solutions 
		\begin{align*}
		~~~~~~|\textbf{u}-\textbf{u}_n|_s \leq \alpha|\widehat{\textbf{G}}-\widehat{\textbf{G}}_n|_{-s}.
		\end{align*}
	\end{lemma}
	\begin{proof}We have
		\begin{align*}
		&\hspace{0.5cm}|\textbf{u}-\textbf{u}_n|_s^2 = \lVert \frak{\frac{s}{2}}(\textbf{u}-\textbf{u}_n) \rVert^2_{\ell^2} \\
		&\hspace{0.5cm}= (\frak{s}(\textbf{u}-\textbf{u}_n),\textbf{u}-\textbf{u}_n)_{\ell^2}\\
		&\hspace{0.5cm}= (\alpha \widehat{\textbf{G}}- \alpha \textbf{u} - \alpha \hat{\textbf{g}}_n + \alpha \textbf{u}_n, \textbf{u}- \textbf{u}_n)_{\ell^2}\\
		&\hspace{0.5cm}=\! \alpha(\widehat{\textbf{G}}- \widehat{\textbf{G}}_n, \textbf{u}- \textbf{u}_n)_{\ell^2} \!-\! \alpha(\textbf{u}-\textbf{u}_n,\textbf{u}-\textbf{u}_n)_{\ell^2}
		\end{align*}
		\begin{equation*}
		\iff |\textbf{u}-\textbf{u}_n|_s^2 + \alpha |\textbf{u}-\textbf{u}_n|_0^2 = \alpha (\widehat{\textbf{G}}-\widehat{\textbf{G}}_n,\textbf{u}- \textbf{u}_n)_{\ell^2}.
		\end{equation*}
		For $s,r \in \R$  with  $s\leq r$ we have the estimate
		\begin{equation}
		\hspace{1cm}|\widehat{\textbf{V}}|_s \leq |\widehat{\textbf{V}}|_r \label{cool}
		\end{equation}
		and therefore
		\begin{align*}
		\hspace{0.8cm}|\textbf{u}-\textbf{u}_n|_s^2 + \alpha |\textbf{u}-\textbf{u}_n|_0^2 \leq (1+\alpha)|\textbf{u}-\textbf{u}_n|_s^2  \\
		\hspace{0.8cm}\leq (1+\alpha)\alpha|\widehat{\textbf{G}}-\widehat{\textbf{G}}_n|_{-s}|\textbf{u}-\textbf{u}_n|_s,
		\end{align*}
		which implies the assertion.\qed
	\end{proof}
	The discrete solution operator $\mathcal{S}_n$ allows us to define an approximation for the optimization problem $j$. 
	\begin{definition}
		The  discrete optimization problem  $j_n$ with $n \in \N$ is defined as 
		\begin{align}\label{endlichbi}
		&\hspace{1cm}\min_{(s,\alpha) \in W} j_n(s,\alpha) \\
		&\hspace{1cm}= \frac{1}{2(2\pi)^{d}}\lVert S_n(s,\alpha)- \widehat{\textbf{U}}_d^n \rVert^2_{l^2} +  \varphi(s,\alpha)\nonumber,
		\end{align}
		with $\widehat{\textbf{U}}_d^n$ is a suitable approximation of $\widehat{\textbf{U}}_d$ in the space $\mathcal{T}_n$.\\
		The associated unique optimal triple for the problem (\ref{endlichbi}) is denoted by $(\bar{s}_n,\bar{\alpha}_n, \mathcal{S}_n(\bar{s}_n, \bar{\alpha}_n))$.
	\end{definition}
	
	\begin{theorem}[Convergence of the projection]\label{konproj}
	Let $g, g_n \in \roomhszero{s_2}$ be and $u_d, u_d^n \in \roomhszero{s_3}$  with $s_2,s_3 > 0$ for the continous and discrete optimization problem~(\ref{func1red}) respectively (\ref{endlichbi}). The functions $g_n, u_d^n$ are defined as the projections  of $g$ and $u_d$ in $\mathcal{T}_n$.\\
	If $ \lVert g \rVert$ and $\lVert u_d \rVert$ are sufficient small, then there exists a constant $\kappa > 0$, such that the Hessian matrix
	\begin{equation}\label{secondderivative}
	\hspace{1cm}\nabla^2 j(s,\alpha)  \succeq \kappa \textbf{I}
	\end{equation}
	for all $(s,\alpha) \in W$.\\
	The associated optimal triples are denoted by \\
	$\hspace{1cm}(\bar{s}, \bar{\alpha}, \bar{\textbf{u}} = \S(\bar{s},\bar{\alpha}))$ respectively \\
	$\hspace{1cm}(\bar{s}_n, \bar{\alpha}_n, \bar{\textbf{u}}_n = \S_n(\bar{s}_n,\bar{\alpha}_n))$. \\
	Then we obtain the following error estimate between the discrete solution
	$(\bar{s}_n,\bar{\alpha}_n)$ and the continuous solution $(\bar{s},\bar{\alpha})$ 
	\begin{align*}
	\Big|\binom{\bar{s}_n}{\bar{\alpha}_n}- \binom{\bar{s}}{\bar{\alpha}}\Big| \leq \frac{C(s_0,\alpha_1)}{\kappa}  (\frac{n}{2})^{-\omega}\max\{|g|_{\omega},|u_d|_{\omega}\},
	\end{align*}
	with $\omega := \min\{s_2,s_3\}$.
\end{theorem}

\begin{proof}
The strong convexity of $j$ and \\
$\nabla j(\bar{s},\bar{\alpha}) = \nabla j_n(\bar{s}_n, \bar{\alpha}_n) = 0$ imply
\begin{align*}
&\kappa \Big|\binom{\bar{s}_n}{\bar{\alpha}_n}- \binom{\bar{s}}{\bar{\alpha}}\Big|^2 \\
&\leq (\nabla j(\bar{s}_n,\bar{\alpha}_n)- \nabla j(\bar{s},\bar{\alpha}))^T(\binom{\bar{s}_n}{\bar{\alpha}_n}- \binom{\bar{s}}{\bar{\alpha}}) \\ 
&= (\nabla j(\bar{s}_n,\bar{\alpha}_n)- \nabla j_n(\bar{s}_n,\bar{\alpha}_n))^T(\binom{\bar{s}_n}{\bar{\alpha}_n}- \binom{\bar{s}}{\bar{\alpha}}).
\end{align*}
As a result we obtain 
\begin{align*}
&\hspace{0cm}\kappa \Big|\binom{\bar{s}_n}{\bar{\alpha}_n}- \binom{\bar{s}}{\bar{\alpha}}\Big|
 \leq |\partial_s j(\bar{s}_n,\bar{\alpha}_n) - \partial_s j_n(\bar{s}_n,\bar{\alpha}_n) |
 \\&\hspace{0cm} + ~|\partial_\alpha j(\bar{s}_n,\bar{\alpha}_n) - \partial_\alpha j_n(\bar{s}_n, \bar{\alpha}_n)|.
\end{align*}
In the following we denote $\mathcal{S}(\bar{s}_n,\bar{\alpha}_n)$ as $\three$.
For the partial derivatives with respect to variable  $s$ we have 
\begin{align*}
&\hspace{0.5cm}|\partial_s j(\bar{s}_n,\bar{\alpha}_n) - \partial_s j_n(\bar{s}_n,\bar{\alpha}_n) |\\
&\hspace{0.5cm}=  \frac{1}{2(2\pi)^d}\bigg|(\three - \widehat{\textbf{U}}_d, \partial_s\three)_{l^2}-(\two - \widehat{\textbf{U}}_d, \partial_s \two)_{\ell^2} \\
&\hspace{0.5cm}+ (\three,\three - \widehat{\textbf{U}}_d)_{\ell^2} -  (\partial_s \two,\two - \widehat{\textbf{U}}_d)_{\ell^2}\bigg|.  
\end{align*} 
Because of the anti-symmetry of the $\ell^2$-scalar product we only consider the first two summands. We have
\begin{align*}
&\bigg| (\three - \widehat{\textbf{U}}_d, \partial_s \three)_{\ell^2} - (\two - \widehat{\textbf{U}}_d^n, \partial_s \two)_{\ell^2}\bigg| \leq\\
& \bigg| (\three - \widehat{\textbf{U}}_d, \partial_s \three)_{\ell^2} - (\two 
- \widehat{\textbf{U}}_d^n, \partial_s \three)_{\ell^2}\bigg| +\\ &\bigg|(\two - \widehat{\textbf{U}}_d^n, \partial_s \three)_{\ell^2} - (\two 
- \widehat{\textbf{U}}_d^n, \partial_s \two)_{\ell^2}\bigg| =: A + B.   
\end{align*}
For the summands $A$ and $B$ we obtain
\begin{align*}
&A \leq \lVert \three - \two \rVert_{\ell^2}\lVert\partial_s \three \rVert_{\ell^2} + \lVert \widehat{\textbf{U}}_d - \widehat{\textbf{U}}_d^n \rVert_{l^2}\lVert \partial_s \three \rVert_{\ell^2}
\end{align*}
and
\begin{align*}
&B \leq \lVert\partial_s\three - \partial_s\two\rVert_{\ell^2}\lVert \two\rVert_{\ell^2} + \\
&\lVert \widehat{\textbf{U}}_d^n \rVert_{\ell^2}\lVert\partial_s \three - \partial_s \two \rVert_{\ell^2}.
\end{align*}
The terms $\lVert\partial_s \three \rVert_{\ell^2},~\lVert \two \rVert_{\ell^2}$ and $\lVert \widehat{\textbf{U}}_d^n \rVert_{\ell^2}$ in summands $A$ and $B$ are bounded with a constant $C$ independent from $n$. The constant $C$ only depends on $s_0, \alpha_1$ and $ \lVert \widehat{\textbf{G}} \rVert_{\ell^2}$.\\
Moreover, we have with Lemma \ref{stabi} and estimate (\ref{cool})
\begin{align*}
\hspace{1cm}\lVert \three  - \two \rVert_{\ell^2} \leq \bar{\alpha}_n\lVert \widehat{\textbf{G}} - \widehat{\textbf{G}}_n \lVert_{\ell^2}.
\end{align*}
Lemma \ref{bepartial} implies
\begin{align*}
&\hspace{0.5cm}\lVert\partial_s\three - \partial_s\two\rVert_{\ell^2}\lVert\two\rVert_{\ell^2} \\
&\hspace{0.5cm}\leq C(s_0, M_{{\alpha_1},1}, \lVert g \rVert)\lVert \widehat{\textbf{G}} - \widehat{\textbf{G}}_n \lVert_{\ell^2}.
\end{align*}
Lemma \ref{proj} and estimate (\ref{cool}) yield 
\begin{align*}
&\hspace{1cm}|\partial_s j(\bar{s}_n,\bar{\alpha}_n) - \partial_s j_n(\bar{s}_n,\bar{\alpha}_n) | \\
&\hspace{1cm}\leq C(s_0,\alpha_1) (\frac{n}{2})^{-\omega} \max\{|g|_{\omega},|u_d|_{\omega}\}.
\end{align*}
with $\omega = \min \{s_2,s_3\}$.
An analogous argumentation for 
$|\partial_\alpha j(\bar{s}_n,\bar{\alpha}_n) - \partial_\alpha j_n(\bar{s}_n, \bar{\alpha}_n)|$ has as result
\begin{align*}
&\hspace{1cm}|\partial_\alpha j(\bar{s}_n,\bar{\alpha}_n) - \partial_\alpha j_n(\bar{s}_n, \bar{\alpha}_n)|  \\
&\hspace{1cm}\leq \tilde{C} (\frac{n}{2})^{-\omega} \max\{|g|_{\omega},|u_d|_{\omega}\},
\end{align*}
with $\tilde{C} > 0$ not depends on $s_0, \alpha_1$.\qed
\end{proof}
In the case of the Fourier interpolation we can prove a similar result, but we need stronger assumptions.
\begin{theorem}[Convergence of the interpolation]~\\
Let $g, g_n \in \roomhszero{s_2}$ and $u_d, u_d^n \in \roomhszero{s_3}$ with $\omega := \min\{s_2,s_3\} > \frac{d}{2}$ be given for the continous and discrete optimization problem (\ref{func1red}) respectively (\ref{endlichbi}). The discrete functions $g_n, u_d^n$ are defined as the Fourier interpolation  of the functions. If $ \lVert g \rVert$ and  $\lVert u_d \rVert$ are sufficient small, then there exists a constant $\kappa > 0$, such that the Hessian matrix
\begin{equation*}
\hspace{1cm}\nabla^2 j(s,\alpha) \succeq \kappa\textbf{I}
\end{equation*}
for all $(s,\alpha) \in W$.
The associated optimal triple we denote by 
$(\bar{s}, \bar{\alpha}, \S(\bar{s},\bar{\alpha}))$ and  $(\bar{s}_n, \bar{\alpha}_n,  \S_n(\bar{s}_n,\bar{\alpha}_n))$.
Then we have following error estimate between the discrete solution 
$(\bar{s}_n,\bar{\alpha}_n)$ and the continous solution $(\bar{s},\bar{\alpha})$
\begin{align*}
&\hspace{1cm}\Big|\binom{\bar{s}_n}{\bar{\alpha}_n}- \binom{\bar{s}}{\bar{\alpha}}\Big| \\
&\hspace{1cm}\leq \frac{C(s_0,\alpha_1,d,\omega)}{\kappa}  (\frac{n}{2})^{-\omega}\max\{|g|_{\omega},|u_d|_{\omega}\}.
\end{align*}
\end{theorem}

\begin{proof}
The proof is analogous to Theorem \ref{konproj} using Lemma \ref{inter}.\qed
\end{proof}
\section{Numerical Experiments}
For our numerical experiments we consider the following specific problem 
\begin{align}
&\min~ j_n(s,\alpha) :=  J_n(s,\alpha, \mathcal{S}_n(s,\alpha))\label{func1red} \\
&= \frac{1}{2(2\pi)^d} \lVert  \mathcal{S}_n(s,\alpha) - \widehat{\textbf{U}}_{d,n} \rVert^2_{\ell^2} + 5\cdot 10^{-5}\varphi_1(s,\alpha) \nonumber \\ 
&\text{with }(s, \alpha) \in [0,0.5]\times[0,250] \text{ and } \nonumber\\
&~~~~~\varphi_1(s,\alpha) = \frac{1}{s(0.5-s)} + \frac{1}{\alpha(250-\alpha)} \nonumber.
\end{align}
We denote $\widehat{\textbf{U}}_{d,n}$ as the interpolated Fourier coefficcients of $u_d$ up to order $n$. The Fourier coefficents of $g$ are defined analogously.
Motivated on embedding arguments in \cite{AntBar17}, we choose for the parameter $s$ values between $0$ and $\frac{1}{2}$. The existence interval for the parameter $\alpha$ is choosen empirically.
A suitable choice of the scaling factor for the function $\varphi_1$ is an important aspect of the optimization problem. If the scaling factor is too large, then the optimization problem is dominated by the function $\varphi_1$. As a consequence, it is mainly optimized with respect to the function $\varphi$ and not after the denoising parameters $s$ and $\alpha$. If the scaling factor is too small, then the convergence properties get worse because of the small influence of the strong convexity constant of the function $\varphi$. In numerical experiments we observe that the runtime of the optimization problem depends on the scaling factor; it increases for a small scaling factor.\\
We solve the restricted optimization problem in MATLAB using an SQP-solver, applied on  
\begin{figure}
	\resizebox{0.9\hsize}{!}{\includegraphics{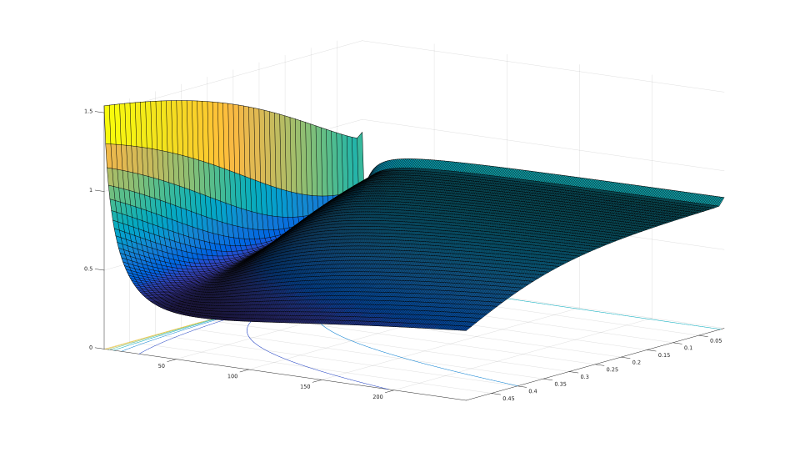}}
	\caption{Energy functional $j_n$ with function $\varphi_1$ and different choices for $s$ and $\alpha $.}
\end{figure} 
\begin{align*}
&\hspace{1cm}\min_{s,\alpha} j_n(s,\alpha)\\
&\hspace{1cm}\text{s.t.~~~} h(s, \alpha) \leq 0
\end{align*}
with $h: \R^2 \rightarrow \R^4$ and
\[\hspace{1cm}h(s, \alpha)=\begin{pmatrix}
- s \\
s - 0.5 \\
- \alpha \\
\alpha -250
\end{pmatrix}.\]
Let $h_i$ for $i =1,2,3,4$ denote the components of the function $h$.
The SQP-problem is the quadratic approximation of the associated Lagrange function 
\begin{align*}
\hspace{1cm}L(s,\alpha,\lambda) = j_n(s,\alpha) + \sum\limits_{i=1}^{4} \lambda_i h_i(s, \alpha)
\end{align*} with
\begin{align*}
&~~~~~\min_{d \in \R^2} \frac{1}{2} d^T \textbf{H}_k d + \nabla j_n(s_k,\alpha_k)^T d\\
&\text{s.t~~} \nabla h_i(s_k,\alpha_k)^T d + h_i(s_k,\alpha_k) \leq 0 ~~~i=1,...,4.
\end{align*}
The matrix  $\textbf{H}_k$  is a positive definit approximation of the Hessian matrix from the Lagrange function $L(s,\alpha,\lambda)$.
The solution $d_k$ of the quadratic program gives us for a suitable step size $\beta_k$
\begin{align*}
\hspace{1cm}\binom{s_{k+1}}{\alpha_{k+1}} = \binom{s_k}{\alpha_k} + \beta_k d_k.
\end{align*} 
Using the fact that
\begin{align*}
\hspace{1cm}\nabla^2_{s,\alpha} L(s, \alpha ,\lambda) = \nabla^2 j_n(s, \alpha)
\end{align*} 
we can argue that the approximation of the Hessian matrix $\nabla^2 j(s,\alpha)$ is positiv definit using the argumentation structure as in the proof of Lemma \ref{posdef}. This implies the well-posedness of the SQP-method. For details of this algorithm we refer to the official documentation of the software library MATLAB.\\
For our numerical examples we obtain our noisy images $g$ from $u_d$, where the additve noise is normally distributed with mean zero and standard deviation $0.15$. For the obtained results, we measure the quality of reconstructions using metrics such as the peak signal-to-noise ratio (PSNR) and structural-similarity-index (SSIM). In Figures ~\ref{boat} and ~\ref{pepper} we illustrate the optimal solution $\bar{u}$ of the discrete optimization problem (\ref{func1red}) for two test images. As expected our model correlates with the standard deviation $\sigma$ as seen in Figures \ref{alpha} and \ref{se}. A higher noise implies that the denoised image $\bar{u}$ has a higher deviation from the noisy image $g$ and a stronger smoothing. 
\begin{figure}
		\fboxsep=2mm
		\centering
		\begin{subfigure}{0.22\textwidth}
			\centering
			\mbox{
				\includegraphics[width=0.9\textwidth]{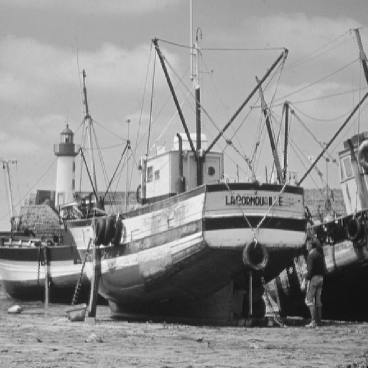}}
			\subcaption{original image $u_d$}
		\end{subfigure}
		\vspace{0.2cm}
		\begin{subfigure}{0.22\textwidth}
			\centering
			\mbox{
				\includegraphics[width=0.9\textwidth]{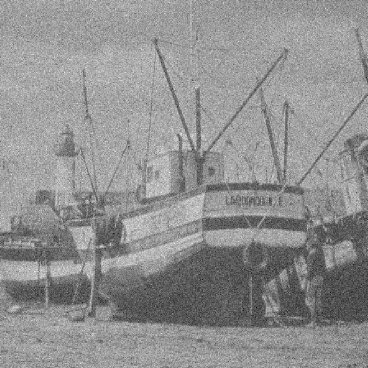}}
			\subcaption{image $g$;~ SSIM: 0.215}
		\end{subfigure}
		\begin{subfigure}{0.22\textwidth}
			\centering
			\mbox{
				\includegraphics[width=0.9\textwidth]{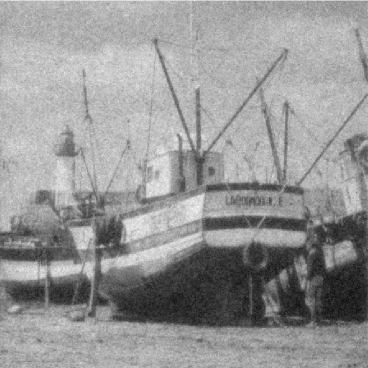}}
			\subcaption{image $u$ with $\bar{s} = 0.471, \bar{\alpha} = 45.38$;~ SSIM: 0.512}
		\end{subfigure}
		\caption{Results for the image "boat". As clearly seen, the model eliminates  the fog based on the normally distributed noise.}
		\label{boat}
	\end{figure}

\begin{figure}[h]
	\fboxsep=2mm
	\centering
	\begin{subfigure}{0.22\textwidth}
		\centering
		\mbox{
			\includegraphics[width=0.9\linewidth]{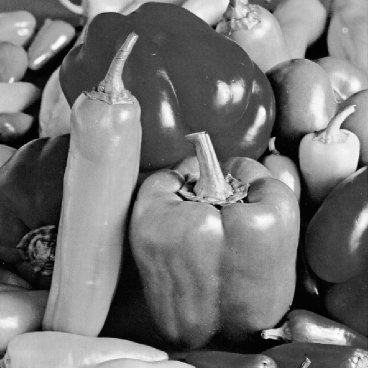}}
		\subcaption{original image $u_d$}
	\end{subfigure}
	\vspace{0.2cm}
	\begin{subfigure}{0.22\textwidth}
		\centering
		\mbox{
			\includegraphics[width=0.9\linewidth]{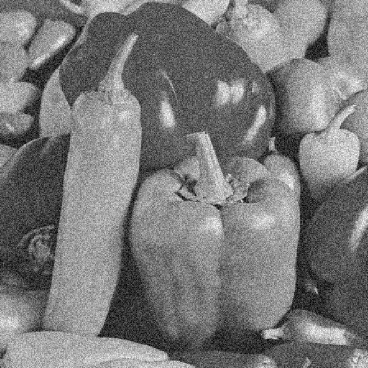}}
		\subcaption{noisy image $g$;~ SSIM: 0.174}
	\end{subfigure}
	\begin{subfigure}{0.22\textwidth}
		\centering
		\mbox{
			\includegraphics[width=0.9\linewidth]{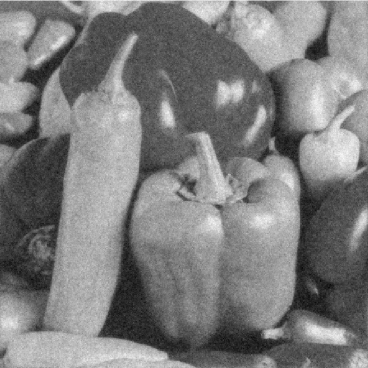}}
		\subcaption{image $u$ with $\bar{s} = 0.475, \bar{\alpha} = 48.93$;~SSIM: 0.495}
	\end{subfigure}
	\caption{In the image "peppers"{} is a significant reduction of noise while maintaining the edges.}
	\label{pepper}
\end{figure}

\begin{figure}
	\begin{center}
		\resizebox{0.5\textwidth}{!}{
			\begin{tikzpicture}
			\begin{axis}[
			xlabel={\textbf{$\sigma$ }},
			ylabel={ \textbf{$\bar{s}$}},
			xmin=0, xmax=0.3,
			ymin=0, ymax=0.5,
			xtick={0,0.05,0.1,0.15,0.2,0.25,0.3},
			ytick={0,0.1,0.2,0.3,0.4,0.5},
			legend pos=north west,
			ymajorgrids=true,
			grid style=dashed,
			]
			
			\addplot[
			color=blue,
			mark=square,
			]
			coordinates {
				(0,0.2)(0.01,0.23)(0.03,0.40)(0.05,0.44)(0.1,0.463)(0.15,0.471)(0.2,0.475)(0.25,0.477)(0.3,0.479)
			};

			\end{axis}
			\end{tikzpicture}}
	\end{center}
	\caption{We study the influence of the noise on the parameter $s$. The numerical results coincide with theoretical considerations that more noise implies a stronger smoothing of the image, i.e a higher value of the parameter $s$.}
	\label{se}
\end{figure}
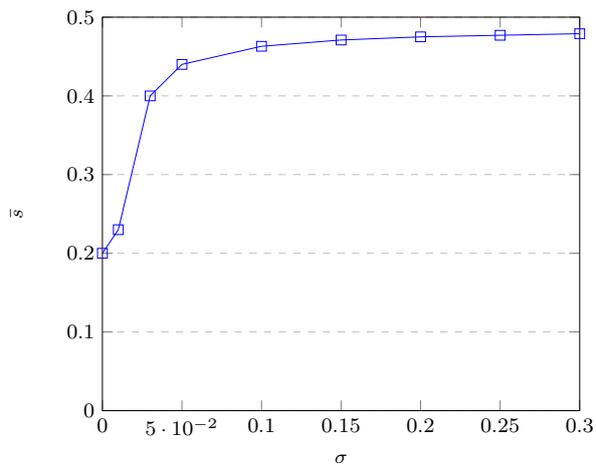

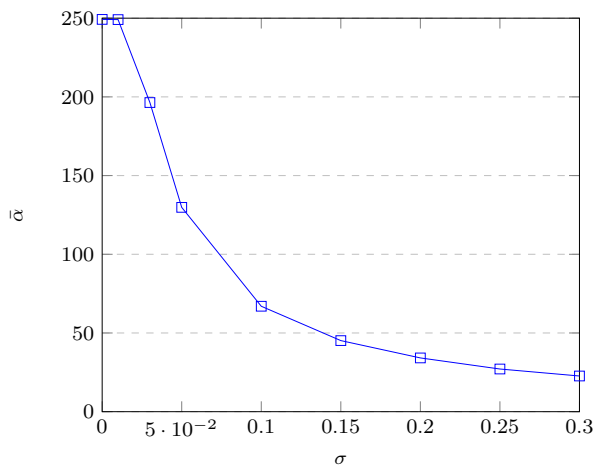
\begin{figure}
	\begin{center}
		\resizebox{0.5\textwidth}{!}{
			\begin{tikzpicture}
			\begin{axis}[
			xlabel={$\sigma$ },
			ylabel={$\bar{\alpha}$},
			xmin=0, xmax=0.3,
			ymin=0, ymax=250,
			xtick={0,0.05,0.1,0.15,0.2,0.25,0.3},
			ytick={0,50,100,150,200,250},
			legend pos=north west,
			ymajorgrids=true,
			grid style=dashed,
			]
			
			\addplot[
			color=blue,
			mark=square,
			]
			coordinates {
				(0,249.22)(0.01,249.1)(0.03,196.4)(0.05,129.8)(0.1,66.98)(0.15,45.16)(0.2,34.15)(0.25,27.11)(0.3,22.67)
			};

			\end{axis}
			\end{tikzpicture}}
	\end{center}
	\caption{ Also in the case of the parameter $\alpha$  the numerical results coincide with theoretical considerations. A higher noise has the result, that the denoised image is more far away from the noisy image, i.e a higher value of the parameter $\alpha$.}
	\label{alpha}
\end{figure}

Furthermore, we compare our model with the ROF model \cite{RudOshFat92}, which consists in minimizing
\begin{align}\label{rofmod}
\hspace{1cm} E(u) = \big| Du \big|_{\torus} + \frac{\alpha }{2}\lVert g - u \rVert^2.
\end{align}
for given $ g \in L^2(\torus, \R)$.
It can be shown, that for $\alpha > 0$ exists a unique minimizer $u \in BV(\torus) \cap L^2(\torus; \R)$.\\
The minimization of the ROF model is done with a gradient flow. The variatonal derivation of total variation  
\begin{align*}
\hspace{1cm} \nabla | Du \big|_{\torus} = {\rm div}\,(\frac{\nabla u}{| \nabla u |}).
\end{align*}
is not differentiable in $0$, so we substitute $| \nabla u(x) |$ with $\sqrt{\varepsilon^2 + | \nabla u(x) | ^2 }$  and $\varepsilon = 0.004 $.\\
For the numerical implementation and test images we refer to \cite{Pey11}. In the ROF model we test $20$ different values for the parameter $\alpha$, all in the range of $[10^{-6},0.2]$ with equidistant distance. Afterwards we choose the parameter $\alpha$, such that we have the highest peak signal-to-noise-ratio in comparison to the reference image $u_d$.

\begin{figure}[h]
	\centering
	\begin{subfigure}{0.22\textwidth}
		\centering
		\mbox{
			\includegraphics[width=0.9\linewidth]{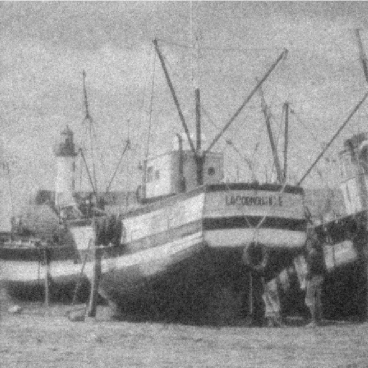}}
		\subcaption{Fractional-Laplacian-model; SSIM: 0.513}
	\end{subfigure}
	\vspace{0.2cm}
	\begin{subfigure}{0.22\textwidth}
		\centering
		\mbox{
			\includegraphics[width=0.9\linewidth]{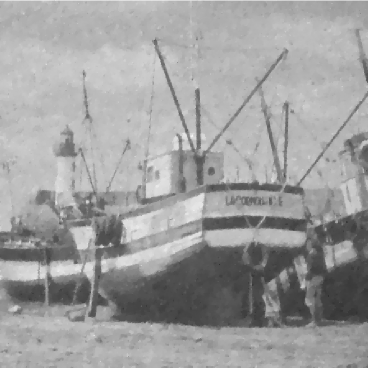}}
		\subcaption{ROF-model; SSIM: 0.735}
	\end{subfigure}
	\caption{In comparison to the fractional Laplacian model has the ROF model  smoother edges. Moreover the ROF model has a better SSIM-value.}
	\label{fig:too}
\end{figure}
\begin{figure}
\begin{center}
	\begin{tabular}{|p{1.5cm}|p{2cm}|p{2cm}|}
		\hline
		Pixels  & Fractional Laplacian  & Regularized ROF  \\ \hline 128 &  0.38s &  6.34s  
		\\ \hline 256 & 1.23 s & 22.28 s
		\\ \hline 512 & 5.9 s & 96.26 s
		\\ \hline 1024 & 25.25 s & 433.59 s
		\\ \hline 2048 & 102.1 s & 2273.59 s 
		\\ \hline
	\end{tabular}
\end{center}
\caption{Runtime comparison between the fractional Laplacian model and the ROF model. Both model show empirically a linear runtime, but the fractional Laplacian model has a reduced computing time by factors 16-22.
We used MATLAB R2015a with CPU i3-3240 and 8 GB RAM.}
\label{runtime}
\end{figure}
As in \cite{GouMor01} mentioned, the space of bounded variation $BV(\torus)$ is insufficient to describe all natural images. Therefore, we look at the image "Baboon"{} as a counterexample. Figure \ref{monkey} shows a part of the coat structure. 
A comparison between the ROF model and the Laplacian model regarding different noise levels shows that the performance of both models highly depends on the choice of the image. But we point out, that our model has a worse performance in comparison of the ROF model, as we see in Figures \ref{psnr} and \ref{ssim}. We compare the runtime of the fractional Laplacian and the ROF model in Figure ~\ref{runtime}. Our model has a significantly lower runtime with a reduction factor of 16 to 22 in time.
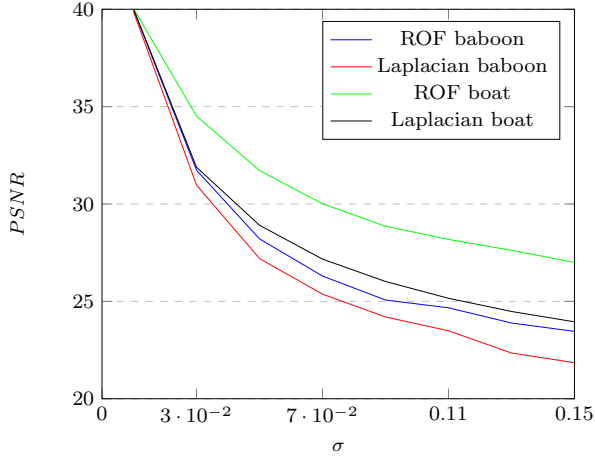
\begin{figure}
	\begin{center}
		\resizebox{0.5\textwidth}{!}{
			\begin{tikzpicture}
			\begin{axis}[
			xlabel={\textbf{$\sigma$ }},
			ylabel={ \textbf{$PSNR$}},
			xmin=0, xmax=0.15,
			ymin=20, ymax=40,
			xtick={0,0.03,0.07,0.11,0.15},
			ytick={20,25,30,35,40},
			legend pos=north west,
			ymajorgrids=true,
			grid style=dashed,
			legend pos= north east
			]
			
			\addplot[
			color=blue
			]
			coordinates {
				(0.01,39.99)(0.03,31.735)(0.05,28.21)(0.07,26.3)(0.09,25.07)(0.11,24.67)(0.13,23.89)(0.15,23.46)
			};
			\addplot[
			color=red
			]
			coordinates {
				(0.01,39.87)(0.03,30.97)(0.05,27.2)(0.07,25.37)(0.09,24.2)(0.11,23.49)(0.13,22.35)(0.15,21.85)
			};
			\addplot[
			color=green
			]
			coordinates {
				(0.01,40.04)(0.03,34.511)(0.05,31.72)(0.07,30.01)(0.09,28.86)(0.11,28.18)(0.13,27.623)(0.15,27)
			};
			\addplot[
			color=black
			]
			coordinates {
				(0.01,40.01)(0.03,31.89)(0.05,28.91)(0.07,27.17)(0.09,26.022)(0.11,25.16)(0.13,24.482)(0.15,23.95)
			};
			\addlegendentry{ROF baboon}
			\addlegendentry{Laplacian baboon}
			\addlegendentry{ROF boat}
			\addlegendentry{Laplacian boat}

			\end{axis}
			\end{tikzpicture}}
	\end{center}
	\caption{The discrepancy of PSNR depends on the choice of the image.}
	\label{psnr}
\end{figure}
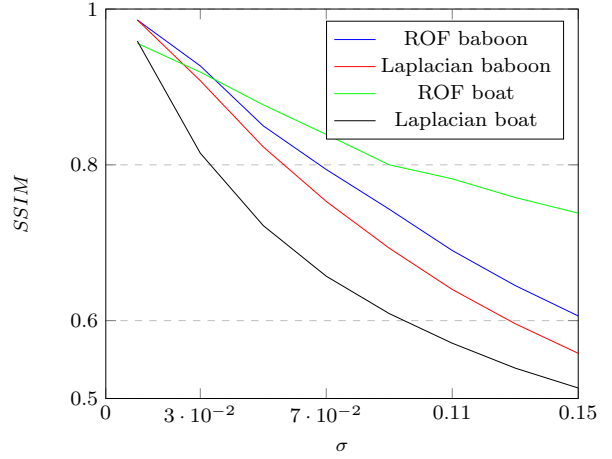
\begin{figure}
	\begin{center}
		\resizebox{0.5\textwidth}{!}{
			\begin{tikzpicture}
			\begin{axis}[
			xlabel={\textbf{$\sigma$ }},
			ylabel={ \textbf{$SSIM$}},
			xmin=0, xmax=0.15,
			ymin=0.5, ymax=1,
			xtick={0,0.03,0.07,0.11,0.15},
			ytick={0.5,0.6,0.8,1},
			legend pos=north west,
			ymajorgrids=true,
			grid style=dashed,
			legend pos= north east
			]
			
			\addplot[
			color=blue
			]
			coordinates {
				(0.01,0.986)(0.03,0.927)(0.05,0.85)(0.07,0.794)(0.09,0.743)(0.11,0.69)(0.13,0.645)(0.15,0.606)
			};
			\addplot[
			color=red
			]
			coordinates {
				(0.01,0.986)(0.03,0.908)(0.05,0.823)(0.07,0.753)(0.09,0.693)(0.11,0.64)(0.13,0.596)(0.15,0.558)
			};
			\addplot[
			color=green
			]
			coordinates {
				(0.01,0.956)(0.03,0.919)(0.05,0.877)(0.07,0.839)(0.09,0.8)(0.11,0.782)(0.13,0.758)(0.15,0.738)
			};
			\addplot[
			color=black
			]
			coordinates {
				(0.01,0.959)(0.03,0.815)(0.05,0.722)(0.07,0.657)(0.09,0.609)(0.11,0.571)(0.13,0.539)(0.15,0.5135)
			};
			\addlegendentry{ROF baboon}
			\addlegendentry{Laplacian baboon}
			\addlegendentry{ROF boat}
			\addlegendentry{Laplacian boat}
			\end{axis}
			\end{tikzpicture}}
	\end{center}
	\caption{Also for the SSIM-value the discrepancy depends on the choice of the image.}
	\label{ssim}
\end{figure}
\begin{figure}[h]
	\centering
	\begin{subfigure}{0.22\textwidth}
		\centering
		\mbox{
			\includegraphics[width=0.9\linewidth]{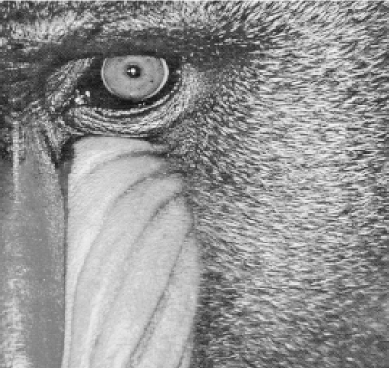}}
		\subcaption{original image $u_d$}
	\end{subfigure}
	\vspace{0.5cm}
	\begin{subfigure}{0.22\textwidth}
		\centering
		\mbox{
			\includegraphics[width=0.9\linewidth]{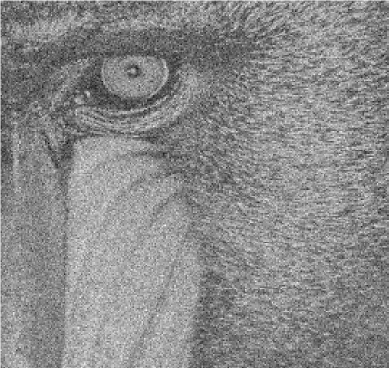}}
		\subcaption{noisy image $g$}
	\end{subfigure}
	\begin{subfigure}{0.22\textwidth}
		\centering
		\mbox{
			\includegraphics[width=0.9\linewidth]{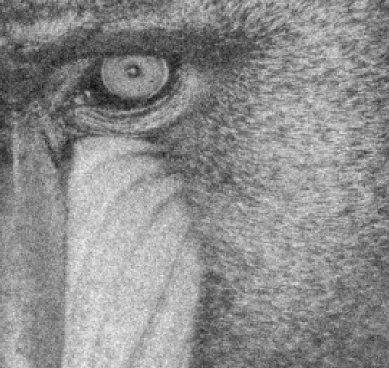}}
		\subcaption{Fractional-Laplacian-model; SSIM: 0.552}
	\end{subfigure}
	\begin{subfigure}{0.22\textwidth}
		\centering
		\mbox{
			\includegraphics[width=0.9\linewidth]{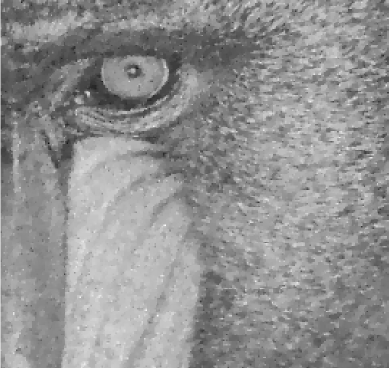}}
		\subcaption{ROF-model; SSIM: 0.604}
	\end{subfigure}
	\caption{Detail of the image "Baboon". The results are comparable.}
	\label{monkey}
\end{figure}
\section{Fractional operators in image decomposition}
In the following we derive a novel approach to decompose an image using fractional differential operators.
Based on the idea to decompose an image in a high and low frequency part we consider the functional
\begin{align}
\hspace{0.5cm}I(u,v) &= \nonumber
\frac{1}{2} \lVert\frak{\frac{s_1}{2}}u\rVert^2  + \frac{\alpha}{2} \lVert u + v - g\rVert^2 + \\
&\hspace{0.5cm}\frac{\beta}{2}  \lVert R_{\frac{s_2}{2}}(v) \rVert ^2 \label{thirdfunc}
\end{align}
with $s_1 \geq 0$ and $s_2 \leq 0$.
\subsection{Existence of a solution and solution operators}
The following theorem is the main result of this section.
\begin{theorem}\label{existence}
	For $g \in \ltorus$ we define 
	\begin{align*}
	~~~~~~~~~~\tilde{g} = g - \frac{1}{|\torus|} \int_{\torus}^{} g~dx.
	\end{align*} 
	Then exist $\tilde{u} \in \roomhszero{s_1}$ and $v \in \dot{H}^{s_2}(\torus;\C) \cap \ltoruszero$, such that the solution pair $(\tilde{u},v)$ minimizes (\ref{thirdfunc}).
	Moreover, the solutions  $u := \tilde{u} + \\\frac{1}{|\torus|} \int_{\torus}^{} g$ and  $v$ fulfill  the identity 
	\begin{align*}
	~~~~~\int_{\torus}^{} (u(x) - g(x))~dx = \int_{\torus}^{} v(x)~ dx = 0.
	\end{align*}
	Furthermore, the solution pair $(u,v)$ is unique.
\end{theorem}
\begin{proof}
	Since $I$ is bounded from below, there exists an infimum sequence $(u_n,v_n)$ with
	\begin{align*}
	&~~~~~\lVert\frak{\frac{s_1}{2}}u_n\rVert^2 \leq C, \\
	&~~~~~\lVert u_n + v_n - g \rVert ^2 \leq C,\\
	&~~~~~\lVert R_{\frac{s_2}{2}}(v_n) \rVert ^2 \leq C.
	\end{align*}
	The constant $C > 0$ is independent of $n$.\\
	The fact $ \lVert u_n \rVert^2 \leq 	\lVert\frak{\frac{s_1}{2}}u_n\rVert^2  \leq C $ yields the uniform boundedness of $u_n$ in $\ltorus$ and therefore also the uniform boundedness of $v_n$ in \\$\ltorus$. As a matter of fact, we have
	\begin{align*}
	~~~~~u_n &\rightharpoonup \tilde{u} \text{ in $\roomhszero{s_1}$,}\\
	~~~~~v_n &\rightharpoonup \tilde{v} \text{ in $\ltorus$,}\\
	~~~~~v_n &\rightharpoonup v \text{ in $\dot{H}^{s_2}(\torus)$}.
	\end{align*}
	Using the test function $w \equiv 1$  and the weak $L^2$-convergence, we obtain 
	\begin{align*}
	\hspace{1cm }&\int_{\torus}^{} \tilde{v}(x)~dx = (\tilde{v},w)=  \\
	\hspace{1cm}\lim\limits_{n\rightarrow \infty}&(v_n,w) = \int_{\torus}^{} v_n(x)~dx = 0.
	\end{align*}
	The fact $\tilde{v} \in \roomhshomo{s_2}$ allows us to identify $v$ with $\tilde{v}$.
	Rellich's theorem (Theorem \ref{rellich}) implies the strong convergence of a subsequence with $u_n \rightarrow \tilde{u} \in \ltoruszero$.\\
	The weak lower semicontinuity of $I$ yields 
	\begin{align*}
	\hspace{1cm} I(\tilde{u},v) \leq \liminf_{n\rightarrow \infty} I(u_n,v_n) = \inf_{w,z}I(w,z),
	\end{align*}
	which implies the existence of a solution.
	To prove the uniqueness of the solution, we use the isometry property between 
	$\ltorus$ and $\ell^2(\Z^d)$. With $\hat{g}_0 = \hat{u}_0$ we obtain 
	\begin{align*}
	\hspace{1cm}&I(\hat{\textbf{U}}, \hat{\textbf{V}}) = \sum\limits_{k \in \Z^d \setminus \{0\}}^{} \frac{1}{2}|k|^{2s_1} |\hat{u}_k|^2  \\ 
	\hspace{1cm}&+ \frac{\alpha}{2} |\hat{u}_k + \hat{v}_k - \hat{g}_k|^2 + \frac{\beta}{2}|k|^{2s_2}|\hat{v}_k|^2.
	\end{align*}
	Because of the uniqueness of Fourier coefficients we can differentiate for arbitrary 
	$k \in \Z^d \setminus \{0\}$  regarding  $\hat{v}_k$ and $\hat{u}_k$. In the minimum of  $I(\hat{U}, \hat{V})$ we get
	\begin{align}
	\hspace{1cm }0 &= \alpha(\hat{u}_k + \hat{v}_k - \hat{g}_k) + \beta |k|^{2s_2}\hat{v}_k \label{der_v},\\
	\hspace{1cm}0 &= |k|^{2s_1} \hat{u}_k + \alpha(\hat{u}_k + \hat{v}_k - \hat{g}_k)  \label{der_u}.
	\end{align}
	for arbitrary $k \in \Z^d \setminus \{0\}$.
	Combining (\ref{der_u}) and (\ref{der_v}) imply
	\begin{align}
	\hspace{1cm} \hat{u}_k = \frac{\beta |k|^{2s_2}}{|k|^{2s_1}} \hat{v}_k ~\text{for } k \in \Z^d \setminus \{0\}. \label{another_u}
	\end{align} 
	Substituting (\ref{another_u}) in (\ref{der_v}), we obtain
	\begin{align*}
	\hspace{1cm}&\alpha(\frac{\beta |k|^{2s_2}}{|k|^{2s_1}} \hat{v}_k + \hat{v}_k - \hat{g}_k) + \beta |k|^{2s_2}\hat{v}_k = 0\\
	\hspace{1cm}&\Leftrightarrow ~~\hat{v}_k = \frac{\alpha \hat{g}_k}{\alpha\beta|k|^{2(s_2-s_1)} + \alpha+ \beta |k|^{2s_2}},
	\end{align*}
	which implies uniquness.\qed
\end{proof}
\begin{definition}
	 Let $Y := [s_0, s_3] \times [\alpha_0, \alpha_1] \times [s_4, s_5] \\
	 \times [\beta_0, \beta_1]$ with $s_3 > s_0 \geq 0$ and $-1  \leq s_4 < s_5 \leq 0$. Define the solution operators $\mathcal{S}_1, \mathcal{S}_2: Y \rightarrow \ell^2(\Z^d)$ as
	\begin{align}
	&~~~~~~~~\mathcal{S}_1(s_1,\alpha,s_2,\beta)= \nonumber\\
	&\hspace{0.8cm}\Bigg [ \frac{\alpha\beta |k|^{2(s_2-s_1)} \hat{g}_k}{\alpha\beta|k|^{2(s_2-s_1)} + \alpha + \beta |k|^{2s_2}} \Bigg ]_{k \in \Z^d \setminus \{0\}} \label{op1}
	\end{align} 
	and
	\begin{align}
	&~~~~~~~~\mathcal{S}_2(s_1,\alpha,s_2,\beta) = \nonumber\\ 
	&\hspace{0.8cm}\Biggl[ \frac{\alpha \hat{g}_k}{\alpha\beta|k|^{2(s_2-s_1)} + \alpha + \beta |k|^{2s_2}}\Biggl]_{k \in \Z^d\setminus\{0\}}. \label{op2}
	\end{align}
\end{definition}
\subsection{Relation to other image models}
In the following we study the behavior in the limiting case when the regularization parameters $\alpha$ and $\beta$ tend to infinity.
\begin{lemma}\label{konalpha}
	Let $\beta > 0$ be and  $g \in \ltoruszero$. With $\alpha_n$ we denote an increasing, positive sequence, such that $\lim\limits_{n \rightarrow} \alpha_n = \infty$. The pair $(u_{\alpha_n}, v_{\alpha_n})$ is the unique minimum (\ref{thirdfunc}) for the specific $\alpha_n$. \\
	The sequence   $(u_{\alpha_n}, v_{\alpha_n})$ is bounded and a subsequence converges weakly  to $(u_0,g-u_0)$, with $u_0$ is the unique minimizer of 
	\begin{align}
	\hspace{0.2cm} I(u) := \frac{1}{2} \lVert\frak{\frac{s_1}{2}}u\rVert^2 + \frac{\beta}{2}  \lVert R_{\frac{s_2}{2}}(u-g) \rVert ^2 \label{exactfunc1}.
	\end{align}
\end{lemma}
\begin{proof}
	The proof follows \cite{AujGilChaOsh05}.
	The existence of a solution $(u_{\alpha_n}, v_{\alpha_n})$ for the specific $\alpha_n$ follows from  (\ref{existence}). Moreover, we have
	\begin{align*}
	\hspace{1cm } I(u_{\alpha_n}, v_{\alpha_n}) \leq I(0,g) = \frac{\beta}{2} \lVert R_{\frac{s_2}{2}}(g) \rVert^2,
	\end{align*}
	which implies the boundedness of the sequence \\$(u_{\alpha_n}, v_{\alpha_n})$ independent of  $n\in \N$. Furthermore, this implies the weak convergence of a subsequence to $u_0$ in $\roomhszero{s_1}$ respectivly $v_0$ in $\roomhshomo{s_2} \cap \ltoruszero$.
	The estimate 
	\begin{align*}
	\lVert u_n + v_n - g\rVert^2 \leq \frac{\beta}{2\alpha_n} \lVert R_{\frac{s_2}{2}}(g) \rVert ^2 
	\end{align*} 
	for all $n \in \N$ guarantees 
	\begin{align*}
	\hspace{1cm} \lVert u_0  + v_0 - g\rVert^2 &= 0,\\
	\hspace{1cm} u_0(x) + v_0(x) &= g(x) ~~\text{a. e.} 
	\end{align*}
	in the limiting case $n  \rightarrow \infty$.
	For arbitrary  $u \in \roomhszero{s_1}$ we have
	\begin{align*}
	\hspace{0.5cm}&\frac{1}{2} \lVert\frak{\frac{s_1}{2}}u\rVert^2 + \frac{\alpha_n}{2} \lVert u + (g-u) - g \rVert^2 \\
	&+ \frac{\beta}{2}  \lVert R_{\frac{s_2}{2}}(u-g) \rVert ^2\\
	&\geq \frac{1}{2} \lVert\frak{\frac{s_1}{2}}u_{\alpha_n}\rVert^2 + \frac{\alpha_n}{2} \lVert u_{\alpha_n} + v_{\alpha_n} - g\rVert^2 \\
	&+ \frac{\beta}{2}  \lVert R_{\frac{s_2}{2}}(v_{\alpha_n}) \rVert ^2\\
	& \geq \frac{1}{2} \lVert\frak{\frac{s_1}{2}}u_{\alpha_n}\rVert^2 +  \frac{\beta}{2}  \lVert R_{\frac{s_2}{2}}(v_{\alpha_n}) \rVert ^2 
	\end{align*}
	for all $ n \in \N$. The weak lower semicontinuity of the functional yields
	\begin{align*}
	&\hspace{1cm}\frac{1}{2} \lVert\frak{\frac{s_1}{2}}u_0 \rVert^2 +  \frac{\beta}{2}  \lVert R_{\frac{s_2}{2}}(g - u_0) \rVert ^2 \\
	&\leq \liminf_{n \rightarrow \infty} \frac{1}{2} \lVert\frak{\frac{s_1}{2}}u_{\alpha_n}\rVert^2 +  \frac{\beta}{2}  \lVert R_{\frac{s_2}{2}}(v_{\alpha_n}) \rVert ^2. 
	\end{align*}
	From this we conclude that $(u_0, g - u_0)$ is the unique minimizer of  (\ref{exactfunc1}).\qed\\
\end{proof}
\begin{figure}[h]
	\fboxsep=2mm
	\centering
	\begin{subfigure}{0.22\textwidth}
		\centering
		\mbox{
			\includegraphics[width=0.9\linewidth]{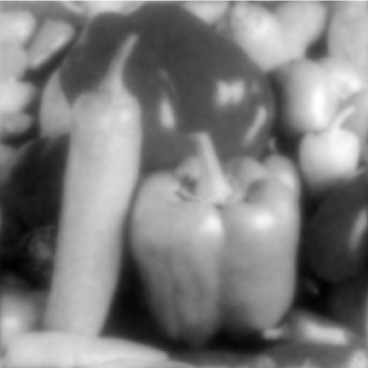}}
		\subcaption{$u$ with $\alpha = 0.001$}
	\end{subfigure}
	\vspace{1cm}
	\begin{subfigure}{0.22\textwidth}
		\centering
		\mbox{
			\includegraphics[width=0.9\linewidth]{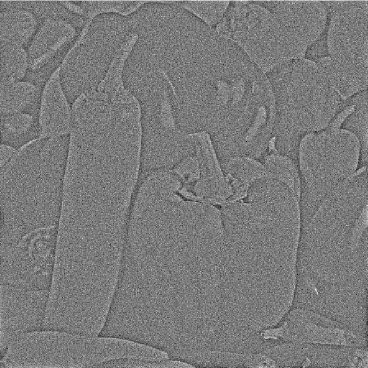}}
		\subcaption{$v$ with $\alpha = 0.001$}
	\end{subfigure}
	\begin{subfigure}{0.22\textwidth}
		\centering
		\mbox{
			\includegraphics[width=0.9\linewidth]{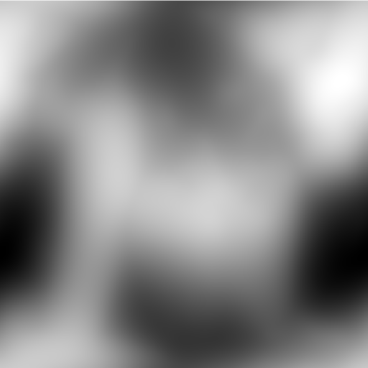}}
		\subcaption{$u$ with $\alpha = 1000$}
	\end{subfigure}
	\begin{subfigure}{0.22\textwidth}
		\centering
		\mbox{
			\includegraphics[width=0.9\linewidth]{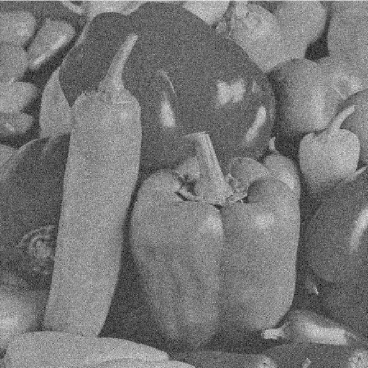}}
		\subcaption{$v$ with $\alpha = 1000$}
	\end{subfigure}
	\caption{Image components with fixed $s_1=0.2, s_2 = -1, \beta =1$ and different choices for parameter $\alpha$ for the noise-free image. The choice of the parameter $\alpha$ has a strong impact on the decomposition.}
	\label{changealpha}
\end{figure}
We can prove a similiar result in the limiting case when $\beta$ tends to infinity.
\begin{lemma}
	Let $ g \in \ltorus$ be and  $\alpha > 0$. With $\beta_n$ we denote an increasing, positive sequence, such that $\lim\limits_{n \rightarrow} \beta_n = \infty$. The pair is the unique minimum $(u_{\beta_n}, v_{\beta_n})$ for this specific $\beta_n$. The sequence  $(u_{\beta_n}, v_{\beta_n})$ is bounded and a subsequence converges weakly to $(u_0,0)$, with $u_0$ is minimizer of (\ref{funcbartels}).
\end{lemma}
\begin{proof}
	Without loss of generality we assume, that $\int_{\torus}^{} g(x)~dx = 0$. The existence of a solution \\$(u_{\beta_n}, v_{\beta_n})$ follows directly from (\ref{existence}). This yields
	\begin{align*}
	\hspace{1cm}I(u_{\beta_n}, v_{\beta_n}) \leq I(0,0) = \frac{\alpha }{2} \lVert g \rVert^2
	\end{align*}
	for all $n \in \N$. As a consequence we have the boundedness of the sequence $(u_{\beta_n}, v_{\beta_n})$, which implies the existence of weakly convergent subsequences to 
	$u_0$ in $\roomhszero{s_1}$ respectively $v_0$ in $\roomhshomo{s_2} \cap \ltoruszero$.
	We obtain
	\begin{align*}
	\hspace{1cm}\lVert R_{\frac{s_2}{2}}(v_{\beta_n})\rVert^2 \leq \frac{\alpha}{2\beta_n} \lVert g \rVert^2 = I(0,0)~\rightarrow 0 
	\end{align*}
	for $n \rightarrow \infty$ and therefore
	\begin{align*}
	\hspace{1cm} v_{\beta_n} \rightarrow 0 = v ~~\text{in~} \roomhshomo{s_2},
	\end{align*}
	i.e. $\lim\limits_{n \rightarrow \infty} \sum\limits_{k \in \Z^d \setminus\{0\}}^{} |k|^{2s_2}|\hat{v}_{\beta_n,k}|^2 = 0$.\\
	To prove the weak convergence of $v_{\beta_n}$ to 0 in \\$\ltoruszero$, we choose an arbitrary function $w \in \ltoruszero$.\\
	The function $R_{s_2}(w) = \frac{1}{(2\pi)^d} \sum\limits_{ k \in \Z^d \setminus \{0\}}^{} |k|^{2s_2} \hat{w}_k \varphi^k $ as an element in $\ltoruszero$ is well-defined and we obtain in the limiting case $n \rightarrow \infty$
	\begin{align*}
	&\hspace{0.5cm}(v_0, R_{s_2}(w))_{\ltoruszero} \leftarrow (v_{\beta_n}, R_{s_2}(w))_{\ltoruszero} \\
	&\hspace{1cm}= \frac{1}{(2\pi)^d} \sum\limits_{k \in \Z^d\setminus\{0\}}^{} |k|^{2s_2}\hat{v}_{\beta_n,k}\overline{\hat{w}_k} \\
	&\hspace{1cm}= \frac{1}{(2 \pi)^d}(v_{\beta_n}, w)_{\roomhshomo{s_2}} \rightarrow 0
	\end{align*}
	and therefore $v_0 \rightharpoonup  0 $ in $\ltoruszero$.\\
	Using the minimizer property of the sequence \\
	$(u_{\beta_n}, v_{\beta_n})$ we have for arbitrary $u \in \roomhszero{s_1}$ 
	\begin{align*}
	&\frac{1}{2} \lVert\frak{\frac{s_1}{2}}u\rVert^2 + \frac{\alpha}{2} \lVert u  - g \rVert^2 = \frac{1}{2} \lVert\frak{\frac{s_1}{2}}u\rVert^2 \\
	&\hspace{1cm}+ \frac{\alpha}{2} \lVert u + 0  - g \rVert^2 + \frac{\beta_n}{2}\lVert R_{\frac{s_2}{2}}(0) \rVert ^2\\
	&\hspace{1cm}\geq \frac{1}{2} \lVert\frak{\frac{s_1}{2}}u_{\beta_n}\rVert^2 + \frac{\alpha}{2} \lVert u_{\beta_n} + v_{\beta_n} - g\rVert^2 \\
	&\hspace{1cm}+ \frac{\beta_n}{2}  \lVert R_{\frac{s_2}{2}}(v_{\beta_n}) \rVert ^2
	\end{align*}
	for all $ n \in \N$.\\
	Because of the weak lower semicontinuity we get
	\begin{align*}
	&\hspace{1cm}\frac{1}{2} \lVert\frak{\frac{s_1}{2}}u_0 \rVert^2 +  \frac{\alpha}{2}  \lVert u_0 - g \rVert ^2 \\
	& \leq \liminf_{n \rightarrow \infty} \frac{1}{2} \lVert\frak{\frac{s_1}{2}}u_{\beta_n}\rVert^2 +  \frac{\alpha}{2}  \lVert u_{\beta_n} + v_{\beta_n} - g \rVert ^2 \\
	&\hspace{1cm}+ \frac{\beta_n}{2}  \lVert R_{\frac{s_2}{2}}(v_{\beta_n}) \rVert ^2.
	\end{align*}
	This shows that $u_0$ is minimizer of (\ref{funcbartels}).\qed
\end{proof}
\section{Image decomposition and numerical experiments}
	
	We first show an error estimate between the continous and discrete solution for fixed, but arbitrary parameters.
	\begin{lemma}
		Let $s_1,\alpha,s_2,\beta$ be fixed, but arbitrarily 
		chosen. The associated pair $(\textbf{u} = \mathcal{S}_1(s_1,\alpha,s_2,\beta), \\ \textbf{v} = \mathcal{S}_2(s_1,\alpha,s_2,\beta)$ is a minimizer of (\ref{thirdfunc}). The pair  $(\textbf{u}_n = \mathcal{S}_{1,n}(s_1,\alpha,s_2,\beta), \textbf{v}_n = \mathcal{S}_{2,n}(s_1,\alpha,s_2,\beta))$ is the solution of the associated discrete problem in the trigonometric space $\mathcal{T}_n$. \\
		Then we obtain the error estimate
		\begin{align*}
		&\hspace{0.5cm}|\textbf{u}-\textbf{u}_n|^2_{s_1} + \beta |\textbf{v}_n - \textbf{v}|^2_{s_2} + \frac{\alpha}{2}\lVert \textbf{u}_n + \textbf{v}_n - \textbf{u} - \textbf{v} \rVert_{\ell^2} ^2 \\
		&\hspace{0.5cm}\leq \frac{\alpha}{2} \lVert \widehat{\textbf{G}} - \widehat{\textbf{G}}_n \rVert_{\ell^2}^2.
		\end{align*}
	\end{lemma}
	\begin{proof}
		The solution pair $(\textbf{u},\textbf{v})$ fulfills the Euler-Lagrange equation
		\begin{align*}
		&\frak{s_1} \textbf{u} + \alpha ( \textbf{u} + \textbf{v}  - \widehat{\textbf{G}}) =  0  \text{ in $\ell^2(\Z^d)$}\\
		\text{and}~~~&\\
		&\alpha(\textbf{u} + \textbf{v} -\widehat{\textbf{G}} ) + \beta R_{s_2}(\textbf{v}) = 0  \text{ in $\ell^2(\Z^d)$}.\\
		\end{align*}
		We obtain an analogous Euler-Lagrange equation for the solution pair $(\textbf{u}_n, \textbf{v}_n)$.
		We have
		\begin{align*}
		&\hspace{1cm}|\textbf{u}-\textbf{u}_n|_{s_1} = \\
		&\hspace{1cm}- \alpha(\textbf{u}_n + \textbf{v}_n - \hat{\textbf{g}}_n - \textbf{u} - \textbf{v} + \widehat{\textbf{G}}, \textbf{u}_n - \textbf{u})_{\ell^2}
		\end{align*}
		and 
		\begin{align*}
		&\hspace{1cm}\beta |\textbf{v}_n - \textbf{v}|_{s_2} = \\
		&\hspace{1cm}- \alpha (\textbf{u}_n + \textbf{v}_n - \widehat{\textbf{G}}_n - \textbf{u} - \textbf{v} + \widehat{\textbf{G}}, \textbf{v}_n -\textbf{v})_{\ell^2}.
		\end{align*}
		Using the Cauchy-Schwarz and Young \\
		inequality yields
		\begin{align*}
		&|\textbf{u}-\textbf{u}_n|_{s_1} + \beta |\textbf{v}_n - \textbf{v}|_{s_2} = \\
		&\alpha (-\textbf{u}_n - \textbf{v}_n + \widehat{\textbf{G}}_n + \textbf{u} + \textbf{v} - \widehat{\textbf{G}}, \textbf{u}_n + \textbf{v}_n - \textbf{u} -\textbf{v})_{\ell^2}\\
		& = - \alpha \lVert \textbf{u}_n + \textbf{v}_n - \textbf{u} - \textbf{v}\rVert^2_{\ell^2}+ \\
		&~~~~ \alpha (\widehat{\textbf{G}}_n - \widehat{\textbf{G}}, \textbf{u}_n + \textbf{v}_n - \textbf{u} - \textbf{v} )_{\ell^2} \\
		& \leq \frac{-\alpha}{2} \lVert \textbf{u}_n + \textbf{v}_n - \textbf{u} - \textbf{v} \rVert^2_{\ell^2} + \frac{\alpha}{2} \lVert  \widehat{\textbf{G}} - \widehat{\textbf{G}}_n \rVert^2_{\ell^2}.\\
		\end{align*}
		This implies the assertion.\qed
	\end{proof}
	\subsection{Numerical experiments}
	For the numerical experiments we consider the optimization problem
	\begin{align}
	&j_n(s_1,\alpha, s_2, \beta) = \frac{1}{2} \lVert \mathcal{S}_{1,n}(s_1,\alpha,s_2,\beta)  - \hat{\textbf{U}}^n_d \rVert^2_{l^2}             \nonumber \\
	&\hspace{1cm}+ \varphi(s_1,\alpha,s_2, \beta) \label{imatwo}\\
	&\hspace{1cm}\text{s.t. } (s_1,\alpha,s_2,\beta) \in W \nonumber.
	\end{align}
	with the discrete solution operator $\mathcal{S}_{1,n}$. 
	For the convex set $W$ we choose 
	\begin{align*}
	\hspace{1cm} [0,0.5] \times[0.01,10^4] \times [-1,0] \times [0,10^5]
	\end{align*}
	and as the strong convex function $\varphi$
	\begin{align*}
	&\hspace{1cm}\varphi(s_1,\alpha,s_2,\beta) = \\
	&\hspace{1cm}\frac{1}{s_1(0.5-s_1)} + \frac{1}{(\alpha-0.01)(10^4-\alpha)}\\
	& \hspace{1cm}+ \frac{1}{-s_2(s_2+1)} + \frac{1}{\beta(10^5-\beta)}.
	\end{align*}
	The scaling factor of the function $\varphi$ is $3\cdot10^{-7}$.\\
	The noise is normal distributed with mean value $\mu = 0$ and standard deviation $\sigma = 0.15$. In all numerical experiments we see a better reconstruction of the original image in comparison to the fractional model (\ref{func1red}), see for example Figure \ref{labelpferd}.

	\begin{figure}
		\fboxsep=2mm
		\centering
		\begin{subfigure}{0.22\textwidth}
			\centering
			\mbox{
				\includegraphics[width=0.9\linewidth]{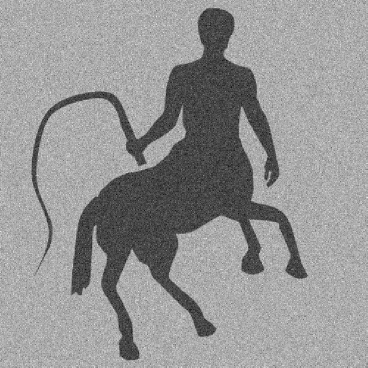}}
			\subcaption{noisy image  $g$\\~}
		\end{subfigure}
		\vspace{0.4cm}
		\begin{subfigure}{0.22\textwidth}
			\centering
			\mbox{
				\includegraphics[width=0.9\linewidth]{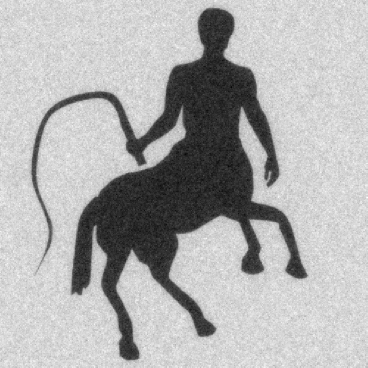}}
			\subcaption{Fractional-Laplacian model (\ref{func1red}); SSIM: 0.356}
		\end{subfigure}
		\begin{subfigure}{0.22\textwidth}
			\centering
			\mbox{
				\includegraphics[width=0.9\linewidth]{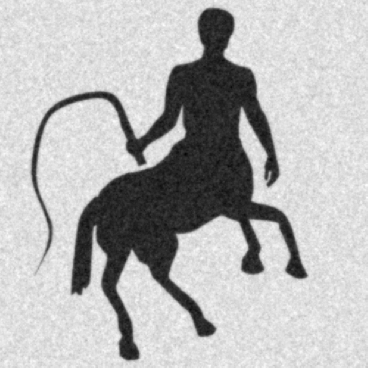}}
			\subcaption{component $u$ of (\ref{imatwo}); SSIM: 0.581}
		\end{subfigure}
		\begin{subfigure}{0.22\textwidth}
			\centering
			\mbox{
				\includegraphics[width=0.9\linewidth]{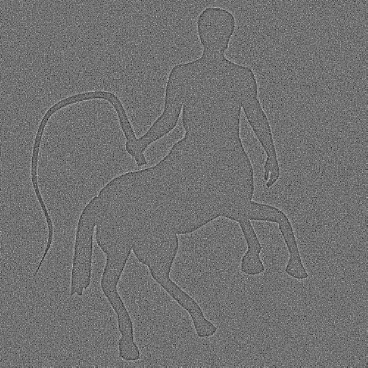}}
			\subcaption{component $v$ of (\ref{imatwo})\newline}
		\end{subfigure}
		\caption{Decompositon of the image "kentaur"{} with optimal parameters $\bar{s_1} = 0.172,~\bar{\alpha} = 9999.7,~\bar{s_2}  = -0.926$ and $\bar{\beta} = 10410$. We see a significantly improvement of the SSIM-value.}
		\label{labelpferd}
	\end{figure}
	\subsection{Comparison with OSV model }
	To compare our model we choose a modification of the ROF model, as shown in \cite{OshSolVes03}. \\There the $L^2$-norm is replaced by the weaker  $H^{-1}$-norm, such that finer details are better \\reconstructed. Therefore, we assume that $g-u = \dive(\vec{v})$ with $\vec{v} \in \ltorus^2$. This yields a unique hodge-decomposition 
	\begin{align*}
	\hspace{1cm} v = \nabla P + \vec{Q}
	\end{align*} 
	with $P \in H^1(\torus)$ and a divergence-free vector field $\vec{Q}$. Using $g- u = \dive(\vec{v}) = \Delta P$ we have $P = \Delta^{-1}(g-u)$. Combining all arguments has as result the convex minimization problem
	\begin{align}
	\inf_{u} E(u) = \int_{\torus}^{} |\nabla u| + \alpha \int_{\torus}^{} |\nabla(\Delta^{-1}(g-u))|^2.\label{seg}
	\end{align}
	which has a solution.\\
	For the numerical implementation we refer to \cite{Pey11}.
	In Figure \ref{two2} and \ref{two3} we compare the Riesz model (\ref{imatwo}) and the OSV model (\ref{seg}). The Riesz model can separate the textural component $v$ much better than the OSV model.
	
	\begin{figure}
		\fboxsep=2mm
		\centering
		\begin{subfigure}{0.22\textwidth}
			\centering
			\mbox{
				\includegraphics[width=0.9\linewidth]{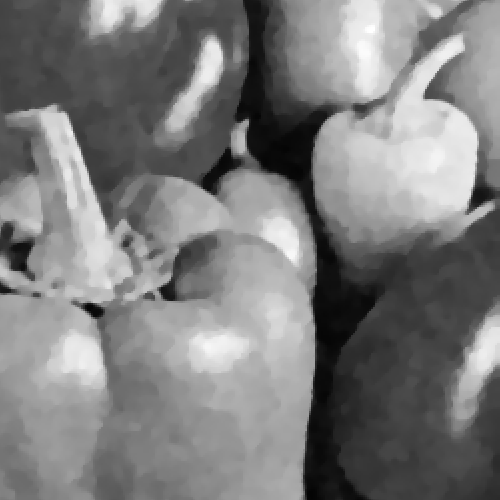}}
			\subcaption{component $u$ of (\ref{seg}); SSIM: 0.798}
		\end{subfigure}
		\vspace{0.4cm}
		\begin{subfigure}{0.22\textwidth}
			\centering
			\mbox{
				\includegraphics[width=0.9\linewidth]{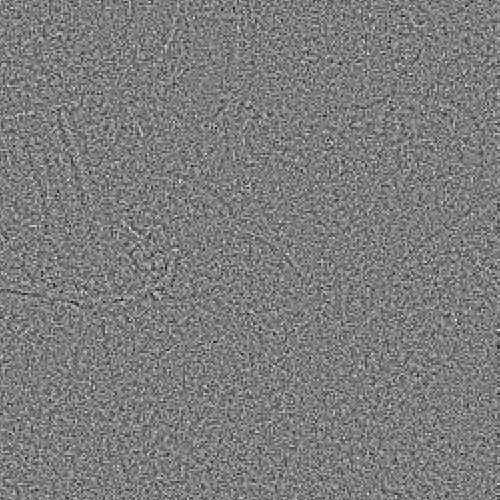}}
			\subcaption{component $g-u$ of (\ref{seg})}
		\end{subfigure}
		\begin{subfigure}{0.22\textwidth}
			\centering
			\mbox{
				\includegraphics[width=0.9\linewidth]{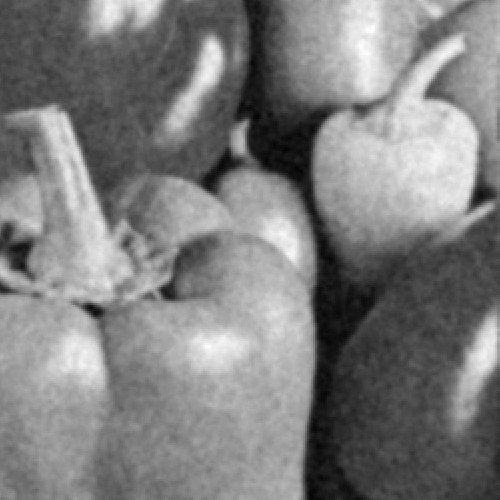}}
			\subcaption{component $u$ of (\ref{imatwo}); SSIM: 0.716}
		\end{subfigure}
		\begin{subfigure}{0.22\textwidth}
			\centering
			\mbox{
				\includegraphics[width=0.9\linewidth]{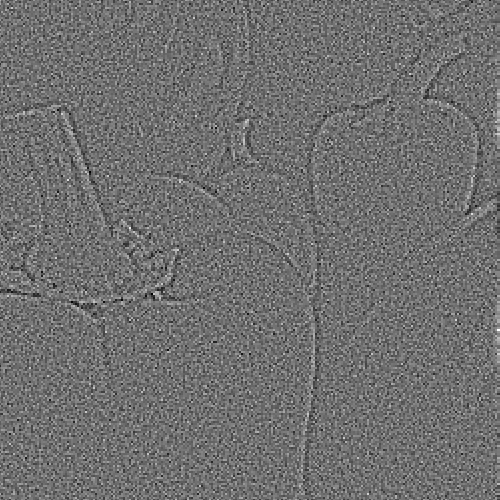}}
			\subcaption{component $v$ of (\ref{imatwo})\newline}
		\end{subfigure}
		\caption{Detail of the image "pepper"{} ($\sigma = 0.1$). The model (\ref{seg}) can not sufficiently distinguish between noise and component $v$.}
		\label{two2}
	\end{figure}
\begin{figure}
	\fboxsep=2mm
	\centering
	\begin{subfigure}{0.22\textwidth}
		\centering
		\mbox{
			\includegraphics[width=0.9\linewidth]{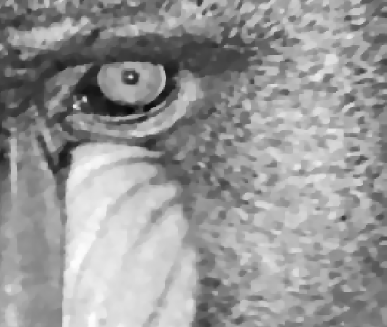}}
		\subcaption{component $u$ of (\ref{seg}); SSIM: 0.553}
	\end{subfigure}
	\vspace{0.4cm}
	\begin{subfigure}{0.22\textwidth}
		\centering
		\mbox{
			\includegraphics[width=0.9\linewidth]{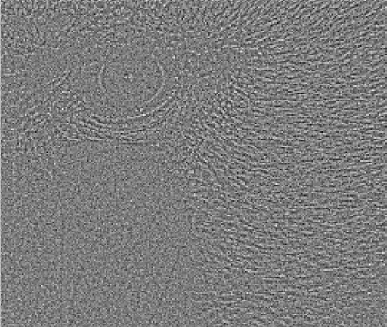}}
		\subcaption{component $g-u$ of (\ref{seg})}
	\end{subfigure}
	\begin{subfigure}{0.22\textwidth}
		\centering
		\mbox{
			\includegraphics[width=0.9\linewidth]{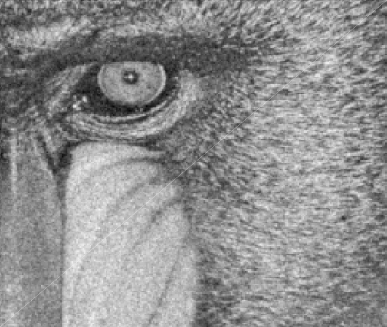}}
		\subcaption{component $u$ of (\ref{imatwo}); SSIM: 0.691}
	\end{subfigure}
	\begin{subfigure}{0.22\textwidth}
		\centering
		\mbox{
			\includegraphics[width=0.9\linewidth]{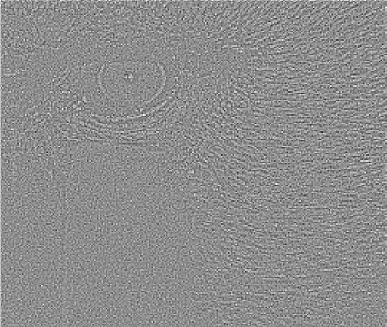}}
		\subcaption{component $v$ of (\ref{imatwo})\newline}
	\end{subfigure}
	\caption{Detail of the image "Baboon"{} ($\sigma = 0.1$).}
	\label{two3}
\end{figure}
\section{Conclusion and outlook}
 In this work we illustrated the possibilities of fractional differential operators in image regularization and decomposition. Working with the Fourier transform we can easily define solution operators. In the case of image regularization the analysis of the solution operator $\mathcal{S}$ allows us to define and analyse a bilevel optimization problem to determine the optimal values for the parameter $s$ and $\alpha$. In contrast to Deep Learning approaches we obtain error estimates and are able to derive an analytical understanding of the problem. The theoretical considerations correlate with the numerical experiments. As an advantage in comparison to the ROF model we can automatically determine the optimal parameters, but we have slightly worse SSIM-values. We point out, that our model has a lower runtime. An open question for further investigations is the choice of a suitable metric to compare the denoised image $u$ and the original image $u_d$. The $L^2$-norm only considers the absolut difference between two images and no structural similarity between them. In the case of image decomposition we introduced a new functional. We proved existence and uniqueness of the solution regarding this model. Moreover, our model can approximate other image decompositon models in the limiting case when $\alpha$ or $\beta$ tends to infinity. The numerical experiments show better results than the fractional Laplacian model. The comparison with the OSV model yields comparable results. Furthermore, we point out, that our model can reconstruct the image component $v$ better than the OSV model.
 The extension to color images is a point of future research. An experimental setup in case of fractional image denoising indicates that the use of quaternionic Fourier transform seems to be the right choice.
\bibliographystyle{spmpsci}
\bibliography{paper_bib}

\begin{thebibliography}{10}
\providecommand{\url}[1]{{#1}}
\providecommand{\urlprefix}{URL }
\expandafter\ifx\csname urlstyle\endcsname\relax
  \providecommand{\doi}[1]{DOI~\discretionary{}{}{}#1}\else
  \providecommand{\doi}{DOI~\discretionary{}{}{}\begingroup
  \urlstyle{rm}\Url}\fi

\bibitem{AntBar17}
Antil, H., Bartels, S.: Spectral approximation of fractional {PDE}s in image
  processing and phase field modeling.
\newblock Comput. Methods Appl. Math. \textbf{17}(4), 661--678 (2017).
\newblock \doi{10.1515/cmam-2017-0039}.
\newblock \urlprefix\url{https://doi.org/10.1515/cmam-2017-0039}

\bibitem{AntKha19}
Antil, H., Khatri, R., et~al.: Bilevel optimization, deep learning and
  fractional laplacian regularization with applications in tomography.
\newblock arXiv preprint arXiv:1907.09605  (2019)

\bibitem{AntOtaSal18}
Antil, H., Ot\'{a}rola, E., Salgado, A.J.: Optimization with respect to order
  in a fractional diffusion model: analysis, approximation and algorithmic
  aspects.
\newblock J. Sci. Comput. \textbf{77}(1), 204--224 (2018).
\newblock \doi{10.1007/s10915-018-0703-0}.
\newblock \urlprefix\url{https://doi.org/10.1007/s10915-018-0703-0}

\bibitem{disantil}
Antil, H., Rautenberg, C.N.: Sobolev spaces with non-{M}uckenhoupt weights,
  fractional elliptic operators, and applications.
\newblock SIAM J. Math. Anal. \textbf{51}(3), 2479--2503 (2019).
\newblock \doi{10.1137/18M1224970}.
\newblock \urlprefix\url{https://doi.org/10.1137/18M1224970}

\bibitem{AujGilChaOsh05}
Aujol, J.F., Gilboa, G., Chan, T., Osher, S.: Structure-texture image
  decomposition{\textemdash}modeling, algorithms, and parameter selection.
\newblock International Journal of Computer Vision \textbf{67}(1), 111--136
  (2006).
\newblock \doi{10.1007/s11263-006-4331-z}.
\newblock \urlprefix\url{https://doi.org/10.1007/s11263-006-4331-z}

\bibitem{BoyVan04}
Boyd, S., Vandenberghe, L.: Convex optimization.
\newblock Cambridge University Press, Cambridge (2004).
\newblock \doi{10.1017/CBO9780511804441}.
\newblock \urlprefix\url{https://doi.org/10.1017/CBO9780511804441}

\bibitem{BoyFab13}
Boyer, F., Fabrie, P.: Mathematical tools for the study of the incompressible
  {N}avier-{S}tokes equations and related models, \emph{Applied Mathematical
  Sciences}, vol. 183.
\newblock Springer, New York (2013).
\newblock \doi{10.1007/978-1-4614-5975-0}.
\newblock \urlprefix\url{https://doi.org/10.1007/978-1-4614-5975-0}

\bibitem{frak1}
Bueno-Orovio, A., Kay, D., Burrage, K.: Fourier spectral methods for
  fractional-in-space reaction-diffusion equations.
\newblock BIT \textbf{54}(4), 937--954 (2014).
\newblock \doi{10.1007/s10543-014-0484-2}.
\newblock \urlprefix\url{https://doi.org/10.1007/s10543-014-0484-2}

\bibitem{GouMor01}
Gousseau, Y., Morel, J.M.: Are natural images of bounded variation?
\newblock SIAM J. Math. Anal. \textbf{33}(3), 634--648 (2001).
\newblock \doi{10.1137/S0036141000371150}.
\newblock \urlprefix\url{https://doi.org/10.1137/S0036141000371150}

\bibitem{LiuSch19}
Liu, P., Sch{\"o}nlieb, C.B.: Learning optimal orders of the underlying
  euclidean norm in total variation image denoising.
\newblock arXiv preprint arXiv:1903.11953  (2019)

\bibitem{cguys}
Liu, Q., Zhang, Z., Guo, Z.: On a fractional reaction-diffusion system applied
  to image decomposition and restoration.
\newblock Comput. Math. Appl. \textbf{78}(5), 1739--1751 (2019).
\newblock \doi{10.1016/j.camwa.2019.05.030}.
\newblock \urlprefix\url{https://doi.org/10.1016/j.camwa.2019.05.030}

\bibitem{Nes04}
Nesterov, Y.: Introductory Lectures on Convex Optimization.
\newblock Springer {US} (2004).
\newblock \doi{10.1007/978-1-4419-8853-9}.
\newblock \urlprefix\url{https://doi.org/10.1007/978-1-4419-8853-9}

\bibitem{OshSolVes03}
Osher, S., Sol\'{e}, A., Vese, L.: Image decomposition and restoration using
  total variation minimization and the {$H^{-1}$} norm.
\newblock Multiscale Model. Simul. \textbf{1}(3), 349--370 (2003).
\newblock \doi{10.1137/S1540345902416247}.
\newblock \urlprefix\url{https://doi.org/10.1137/S1540345902416247}

\bibitem{Pey11}
Peyr{\'{e}}, G.: The numerical tours of signal processing.
\newblock Computing in Science {\&} Engineering \textbf{13}(4), 94--97 (2011).
\newblock \doi{10.1109/mcse.2011.71}.
\newblock \urlprefix\url{https://doi.org/10.1109/mcse.2011.71}

\bibitem{RudOshFat92}
Rudin, L.I., Osher, S., Fatemi, E.: Nonlinear total variation based noise
  removal algorithms.
\newblock Phys. D \textbf{60}(1-4), 259--268 (1992).
\newblock \doi{10.1016/0167-2789(92)90242-F}.
\newblock \urlprefix\url{https://doi.org/10.1016/0167-2789(92)90242-F}.
\newblock Experimental mathematics: computational issues in nonlinear science
  (Los Alamos, NM, 1991)

\bibitem{spec}
Saranen, J., Vainikko, G.: Periodic integral and pseudodifferential equations
  with numerical approximation.
\newblock Springer Monographs in Mathematics. Springer-Verlag, Berlin (2002).
\newblock \doi{10.1007/978-3-662-04796-5}.
\newblock \urlprefix\url{https://doi.org/10.1007/978-3-662-04796-5}

\bibitem{Sprekel}
Sprekels, J., Valdinoci, E.: A new type of identification problems: optimizing
  the fractional order in a nonlocal evolution equation.
\newblock SIAM J. Control Optim. \textbf{55}(1), 70--93 (2017).
\newblock \doi{10.1137/16M105575X}.
\newblock \urlprefix\url{https://doi.org/10.1137/16M105575X}

\end{thebibliography}
\end{document}